\documentclass[12pt]{article}
\usepackage{mathtools}
\usepackage{epsf}
\usepackage{amsmath}
\usepackage{amssymb}
\usepackage{amsthm}
\usepackage{centernot}
\usepackage{hyperref}

\theoremstyle{definition}
\newtheorem{definition}{Definition}[section]
\newtheorem{theorem}[definition]{Theorem}
\newtheorem{lemma}[definition]{Lemma}
\newtheorem{proposition}[definition]{Proposition}

\newtheorem{remark}[definition]{Remark}

\theoremstyle{definition}
\newtheorem{example}[definition]{Example}

\title{Einstein-Hilbert Action and on the Gauss-Bonnet Theorem for Riemannian Noncommutative Tori}
\author{Javad Golipour\\
\\
Institute for Research in Fundamental Sciences (IPM), Tehran \footnote{Email: javadgolipour@ipm.ir}
}
\date{\vspace{-5ex}}
\begin{document}
\maketitle

%Instruction of colors:
%\textcolor{Blue}{$\blacktriangleright$} It used for new items.
%\textcolor{Purple}{$\blacktriangleright$} It used for edited items.
%textcolor{Green}{$\blacktriangleright$} It used for remarkable comments.
%textcolor{Red}{$\blacktriangleright$} It used for questions.

%\begin{abstract}
%In this note we want to review the work of Babson and Kozlov who solved the conjecture of Lovasz about the lower bound for chromatic number of graphs.
%\end{abstract}
%...................................................................%
%...................................................................%

\tableofcontents 
\begin{abstract}
	We show the non-positivity of the Einstein-Hilbert action for conformal flat Riemannian metrics. The action vanishes only when the metric is constant flat. This recovers an earlier result of Fathizadeh-Khalkhali in the setting of spectral triples on noncommutative four-torus. Furthermore, computations of the gradient flow and the scalar curvature of this space based on modular operator are given. We also show the Gauss-Bonnet theorem for a parametrized class of non-diagonal metrics on noncommutative two-torus.
\end{abstract}
\section{Introduction}
%spectral invariants by pseudo calculus%%...................................................%
The Gauss-Bonnet theorem on noncommutative torus was proved in the paper of A. Connes and P. Tretkoff \cite{Con3}. The main idea was using pseudodifferential calculus \cite{Gil1} and the idea of uniformization theorem which assumed a correspondence between the the complex structures and conformal classes on noncommutative two-torus. Following that, some other results about geometric invariants in noncommutative setting were published, \cite{Mar1, Con2,Dab1, Mas2, Mas3, Sit1}. 
\\
%approach of connection theory%%...................................................%
To construct a suitable setting for expressing the concepts of differential geometry in noncommutative framework, one needs to define a reliable connection theory and a Riemannian metric. The concept of compatible connection for a $C^{*}-$dynamical system $(A,G, \alpha)$ and in the special case for a $C^{*}-$dynamical system on a noncommutative torus was introduced by Connes \cite{Con1}. Following that work, it has been proved in \cite{Ros1} a Levi-Civita's theorem for irrational rotation algebras restricted to a dense class given by generic matrices of rotations. Besides, due to lack of derivation property on module action of $A_{\theta}^{\infty}$ on canonical derivations, two different spaces used in \cite{Ros1} to represent the vector fields; the space of derivations on $A_{\theta}^{\infty}$ as the Lie algebra of the space $Aut(A_{\theta}^{\infty})$ of $*-$automorphism group of $A_{\theta}^{\infty}$, and the space of free $A_{\theta}^{\infty}-$module generated by canonical derivations which admits the module action, but it is not a Lie algebra.
\\
There are a lot of results according to the spectral triples or connection theory close to the subject of the paper that consider mathematical physical subjects specially in quantum gravity such \cite{Maj1, Con5, Mas3, Sit2}.
\\
The aim of this paper is understanding more about the connection theory and Riemannian geometry of noncommutative torus. In Section \ref{SecEH}, we show the non-positivity of the Einstein-Hilbert action and that its extremum case occurs for a constant flat metric. The same result is proved in \cite{Mas3} in the setting of spectral triples. In Section \ref{SecGB}, we prove the Gauss-Bonnet theorem for a specific case by complex residue and for a parametrized class of non-diagonal metrics on noncommutative torus. This is the first time in noncommutative geometry that non-diagonal metrics have been considered. In Section \ref{SecGBFails} it is shown that the Gauss-Bonnet theorem in the noncommmutative setting violates in general. 
\\
%applications in mathematical physics%%...................................................%

%...................................................................%
%...................................................................%

\section{Riemannian metrics and Levi-Civita connection}

In this section we recall from \cite{Ros1} a Levi-Civita's theorem for a class of  noncommutative tori.  

Let $\Theta= (\theta_{ij})$ be a real  skew symmetric $n \times n$ matrix.  The noncommutative $n-$torus $A_{\Theta}$ is the universal unital  $C^{*}-$algebra generated by $n$ unitaries  $U_{1},...,U_{n}$ wich satisfy the commutation  relations 
$$U_{j}U_{k} = e^{2 \pi i \theta_{jk}} U_{k}U_{j}.$$ 
For a multi index $k =(k_1, \dots, k_n) \in \mathbb{Z}^n$, let $U^k= U_1^{k_1} \dots U_n^{k_n}$. There is an action of   $\mathbb{R}^n$ on the $C^*$-algebra $A_{\Theta}$ defined by 
$$ \alpha_s(U^k)=e^{i \langle s, k\rangle}U^k,\quad  s\in \mathbb{R}^n.$$
Let $A_{\Theta}^{\infty}$ denote the set of smooth vectors for this action. It is an involutive  dense subalgebra of $A_{\Theta}$. Alternatively it can be described as 
 $$ A_{\Theta}^{\infty} =\left\{\sum_{k \in \mathbb{Z}^n} a_{k}U^k|  \, (a_{k}) \text{ is rapidly decreasing}\right\}.$$
The infinitesimal generators of the action $\alpha$ are the unbounded  $*$-derivations $\delta_j,  j=1, \dots,n$ which satisfy 
$$ \delta_j (U_k)= 2 \pi i \delta_{jk} U_k.$$

Let  $\mathcal{X}_{\Theta}$  denote the the free rank$-n$ left $A_{\Theta}^{\infty}-$module generated by the canonical derivations $\partial_{1},...,\partial_{n}$. Let $\mathcal{D}_{\Theta}$
denote the space of $*$-derivations $\delta: A_{\Theta}^{\infty} \to A_{\Theta}^{\infty}.$ It is real a Lie algebra.   
It is known that for almost all skew-adjoint matrices $\Theta$,  any derivation $\delta \in  \mathcal{D}_{\Theta}$ can be   uniquely  decomposed as  $a_1 \partial_1 +\dots +a_n \partial_n + \delta_0$ where $\delta_0$ is an inner derivation. 

\begin{definition} \label{Riemannianmetric}
A Riemannian metric on $A_{\Theta}^{\infty}$ is a sesquilinear  map $ g =\langle \cdot , \cdot \rangle : \chi_{\Theta} \times \chi_{\Theta} \longrightarrow  A_{\Theta}^{\infty}$ such that the following axioms of  a pre-Hilbert module holds. That is for all  $X, Y \in \mathcal{X}_{\Theta},$ and $a \in A_{\Theta}^{\infty}$ we have\\
(1)  $\langle X, Y\rangle = \langle Y,X\rangle^{*}$  (Hermitian symmetry),\\
(2) $\langle aX,Y\rangle = a \langle X, Y\rangle,$ \\
 (3) $\langle X, X\rangle \geq 0$, and $\langle X, X\rangle =0 $ only if $X=0$, \\
 (4) $\langle \partial_{i}, \partial_{j}\rangle= \langle \partial_{j},\partial_{i}\rangle$ for all $i, j,$ (reality condition). 
\end{definition}

\begin{definition} \label{connection1} 
A connection is a map $\nabla : \mathcal{D}_{\Theta} \times \mathcal{X}_{\Theta} \longrightarrow \mathcal{X}_{\Theta}$ with the properties:
\\$(1)$ $\bigtriangledown$ is $\mathbb{C}$-linear relative to the first variable (this is because $D_{\Theta}$ in contrast to the classical case is no longer $A_{\Theta}^{\infty}$-module). 
\\$(2)$ $\bigtriangledown$ satisfies in the Leibniz rule relative to the second variable. 
\\$(3)$ We choose among other possibilities, $\bigtriangledown_{ada} = a.$, left multiplication by $a$ for each $a \in  A_{\Theta}^{\infty}$ with $\tau(a)=0$. Cf. Remark \ref{uniquenessremark}. 
\\$(4)$ $<\bigtriangledown_{\partial_{i}} \partial_{j}, \partial_{k}>$ is self-adjoint, since we want the covariant derivative of a real vector in a real direction be real-valued.
\end{definition}

A connection $\bigtriangledown$ is called compatible with the Riemannian metric $g$ if for 
all $X, Y \in \mathcal{X}_{\Theta}$ and $Z\in \mathcal{D}_{\Theta}$
$$Z \cdot \langle X, Y\rangle =  \langle  \bigtriangledown_Z  X, Y\rangle +  \langle X, \bigtriangledown_Z Y \rangle$$
and is called torsion free if for any two partial derivatives, 
$ \bigtriangledown_{\partial_{i}}  \partial_{j} = \bigtriangledown_{\partial_{j}}  \partial_{i}$.

%...........................................................%
%...........................................................%

\begin{theorem} \label{Levi-Civita} ( Levi-Civita \cite{Ros1}).
For a generic $\Theta$ and a  Riemannian metric $g$,  there is a unique torsion free compatible connection $\bigtriangledown$.  This is called Levi-Civita connection which is determined, due to the nondegeneracy of the metric, by the formula
\begin{equation}\label{Cristoffell}
<\bigtriangledown_{\partial_{i}} \partial_{j}, \partial_{k}> = \frac{1}{2} \big \{ \partial_{i} \langle \partial_{j}, \partial_{k} \rangle + \langle \partial_{j}   \partial_{i}, \partial_{k}> - \partial_{k} <\partial_{i}, \partial_{j}>  \big \}. 
\end{equation} 
\end{theorem}	

By Theorem \ref{Levi-Civita}, one can define Riemannian curvature for a given Riemannian metric. It is given as in the classical case by the formula
$$R(X,Y):= \bigtriangledown_{Y} \bigtriangledown_{X}- \bigtriangledown_{X} \bigtriangledown_{Y} + \bigtriangledown_{[X,Y]}: X_{\theta} \rightarrow X_{\theta}, \hspace{1 cm} X,Y \in D_{\theta},$$
where we have the symmetry $R_{ijkl}= -R_{jikl}$ and the Bianchi identity $R_{ijkl}+ R_{kijl} + R_{jkil}=0$. However, this tensor does not inherit other symmetries from the classical case. See Remark \ref{nonsymmetrictensorremark}.

%...........................................................%
\begin{remark} \label{uniquenessremark}
The property $(3)$ in Definition \ref{connection1} is necessary for the uniqueness of the compatible and torsion free connection. It is shown in \cite[Proposition 2]{Pet1} that if we drop this condition from Definition \ref{connection1}, we still get for any inner $*-$derivation a compatible and torsion free connection with nontrivial effect on the curvature.

\end{remark}

\begin{example} { \bf Conformal Flat Metric.} \label{NCTorus}
Choosing the conformal flat metric $ g_{ij}= <\partial_{i}, \partial_{j}> = e^{f} \delta_{ij}$, and using the formula (\ref{Cristoffell}) for the connection $<\bigtriangledown_{i} \partial_{j}, \partial_{k}>$, we obtain the curvature tensor 

\begin{equation} \label{NCR1212twodimensional}
R_{1212}= <R(\partial_{1},\partial_{2})\partial_{1}, \partial_{2}>= \big< \big( \bigtriangledown_{2} \bigtriangledown_{1}- \bigtriangledown_{1} \bigtriangledown_{2} \big) \partial_{1}, \partial_{1} \big>
\end{equation}
$$ = - \frac{1}{2}  \big\{ \partial_{2}(\partial_{2}(e^{f}) e^{-f})  + \partial_{1}(\partial_{1}(e^{f}) e^{-f}) \big\} e^{f}.$$

\end{example}

\begin{remark}
If we choose $f$ so that it commutes with its partial derivatives $\partial_{1}$ and $\partial_{2}$,  the curvature tensor (\ref{NCR1212twodimensional}) is simplified to the classical case $- \frac{1}{2} e^{f}\bigtriangleup f  := - \frac{1}{2} e^{f} ( \partial_{1}^{2} + \partial_{2}^{2})f.$ 
\end{remark}

\begin{proposition} \label{G-B}
In Example \ref{NCTorus}, the Gauss-Bonnet theorem holds for any conformal flat metric.
\end{proposition}
\begin{proof}\cite[Proposition 4.1]{Ros1}
We have $R = e^{-2f}(R_{1212}+ R_{2121})$ and since in this case we have also $R_{1212} = R_{2121}$, it implies that
$$\tau(R \sqrt{detg}) = 2 \tau(e^{-2f}R_{1212} e^{f}) =- \tau \big(e^{-2f} \big\{ \partial_{2}(\partial_{2}(e^{f}) e^{-f})  + \partial_{1}(\partial_{1}(e^{f}) e^{-f}) \big\} e^{2f} \big) =0. $$
\end{proof}
%...................................................................%
%...................................................................%
\section{The Einstein-Hilbert action} \label{SecEH}
It was shown in \cite{Mas3} that the Einstein-Hilbert (E-H) action $\int R(g)$ for noncommutative $4-$torus is non-positive and the extremum case occurs only when the metric is constant. In this section we discuss these questions in the setting with the presence of a Levi-Civita connection.  

%...................................................................%
%...................................................................%

\subsection{Einstein-Hilbert action on noncommutative 4-tori} \label{sectionE-H}
We obtain in this subsection similar results as in \cite{Mas3} about the sign and extremum of E-H action. We will also discuss about the same questions for non-conformal diagonal metrics.

\begin{proposition} \label{E-Himproved}
Assume that the noncommutative $4-$torus along with its canonical derivations $\partial_{1}, ..., \partial_{4}$ generating the free projective left $A_{\Theta}$-module  $X_{\Theta}$ are given. Assume also the conformal flat metric on $X_{\Theta}$ is given by $g_{ij} = <\partial_{i}, \partial_{j}>= e^{f} \delta_{ij}, \hspace{1 cm} i,j= 1,...,4$ where $f=f^{*} \in A_{\Theta}$. Then for the associated scalar curvature $R$, we have $\tau(R\sqrt{detg}) \leq 0$ and the extremum case occurs only when the conformal coefficient is constant.
\end{proposition}
\begin{proof}
To obtain the scalar curvature, we first have $R_{ii}= R_{iji}^{j}= e^{-f}R_{ijij}$, $j \neq i$,  $j \in \{1,...,4 \}$. Since we shall show that $R_{ijij} = R_{jiji}$, each of these terms are repeated two times. So we get 
\begin{equation} \label{EHRsymbolic}
R = e^{-f} (R_{11} +R_{22} + R_{33} + R_{44}) = 2e^{-2f} \big \{   R_{1212} + R_{1313} + R_{1414} + R_{2323} + R_{2424} + R_{3434} \big \} 
\end{equation}

We can then apply the Christoffel coefficients for the Riemannian metric given by the Levi-Civita's Theorem \ref{Levi-Civita} to obtain the related Riemann tensors. We obtain for example
\begin{equation*} 
R_{1212}= - \frac{1}{2} \big\{  \partial_{1} \big( (\partial_{1} e^{f}) e^{-f} \big) + \partial_{2} \big( (\partial_{2} e^{f})e^{-f}  \big) \big \}e^{f}  
\end{equation*} 

\begin{equation} \label{EHRtensor}
- \frac{1}{4} (\partial_{3} e^{f}) e^{-f} (\partial_{3} e^{f})- \frac{1}{4} (\partial_{4} e^{f})e^{-f} (\partial_{4} e^{f}),
\end{equation}
and similar relations for other $R_{ijij}$'s. Since the metric and the indices in $R_{1212}$ are symmetric  relative to $1$ and $2$, we will have $R_{1212} = R_{2121}$. Hence according to the relations (\ref{EHRsymbolic}) and (\ref{EHRtensor}) we obtain

\begin{equation} \label{EHRfinal}
\begin{split}
R & = 6 e^{-2f} \Sigma_{i =1}^{4} \big\{ 
- \frac{1}{2}  \partial_{i} \big( (\partial_{i} e^{f}) e^{-f} \big)e^{f}    - \frac{1}{4} (\partial_{i} e^{f}) e^{-f} (\partial_{i} e^{f}) \big\} \\
 & = - \frac{3}{2} e^{-2f} \Sigma_{i =1}^{4} \big\{ 
2 \partial_{i} \big( (\partial_{i} e^{f}) e^{-f} \big)e^{f} + (\partial_{i} e^{f}) e^{-f} (\partial_{i} e^{f}) \big\}.
\end{split}
\end{equation}
Now we compute the Einstein-Hilbert action
\begin{equation} \label{eq1}
\begin{split}
\tau (R e^{2f}) & = - \frac{3}{2} \Sigma_{i =1}^{4}  \tau \big( e^{-2f} \big\{ 2 \partial_{i} \big( (\partial_{i} e^{f}) e^{-f} \big)e^{f} + (\partial_{i} e^{f}) e^{-f} (\partial_{i} e^{f})  \big\}  e^{2f} \big) \\
 & =  - \frac{3}{2} \Sigma_{i =1}^{4}  \tau \big( 2 \partial_{i} \big( (\partial_{i} e^{f}) e^{-f} \big)e^{f} + (\partial_{i} e^{f}) e^{-f} (\partial_{i} e^{f})  \big).
\end{split}
\end{equation}

Besides, the property of integration by parts for $\tau$
$$
\tau \big(\partial_{i} \big( (\partial_{i} e^{f}) e^{-f} \big)e^{f} \big)  =- \tau \big( (\partial_{i} e^{f}) e^{-f} (\partial_{i} e^{f})  \big),
$$
induces 

\begin{equation} \label{E-Hrelation}
\begin{split}
\tau(R e^{2f}) & = \frac{3}{2} \Sigma_{i =1}^{4} \tau \big( (\partial_{i} e^{f}) e^{-f} (\partial_{i} e^{f})  \big)\\
& = \frac{3}{2} \Sigma_{i =1}^{4} \tau \big( (\partial_{i} e^{f}) e^{-f/2} e^{-f/2}(\partial_{i} e^{f})  \big) \\
 & = - \frac{3}{2} \Sigma_{i =1}^{4} \tau \big( (\partial_{i} e^{f}) e^{-f/2} \big( (\partial_{i} e^{f}) e^{-f/2} \big)^{*}  \big) \leq 0,
\end{split}
\end{equation}

where we have used $(\partial_{i} a)^{*} =- \partial_{i} a^{*}$ and   $\tau(a a^{*}) \geq 0$. This proves the first statement of the theorem. In the extremum case if we have $\tau (R e^{2f})= 0$, then we should have

\begin{equation*} 
\begin{split}
 & \tau \big( (\partial_{i} e^{f}) e^{-f/2} \big( (\partial_{i} e^{f}) e^{-f/2} \big)^{*}  \big) = 0, \\
 & (\partial_{i} e^{f}) e^{-f/2} =0, \\
 & \partial_{i} e^{f} =0, \hspace{1 cm} i=1,...,4.
\end{split}
\end{equation*}
So the conformal coefficient $e^{f}$ is constant. This is the second statement of the theorem.

\end{proof}

\begin{remark}
Computing the action functional as Wodzicki residue, it is shown in \cite{Sit1} and \cite{Sit2} that the E-H action for the conformal flat metric up to a constant multiple is

\begin{equation*} 
\begin{split}
\tau(R \sqrt{g}) & =  Wres D_{e^{f}}^{-1} =\Sigma_{i=1}^{4}  \tau (e^{2f} \partial_{i}e^{f}\partial_{i}e^{f} +  e^{f}( \partial_{i}e^{f})e^{f}( \partial_{i}e^{f}) ) \\
 & = - \Sigma_{i=1}^{4} \tau (e^{f} \partial_{i}e^{f}(e^{f}\partial_{i}e^{f})^{*}) - \Sigma_{i=1}^{4}\tau(  e^{f/2}( \partial_{i}e^{f})e^{f/2}(e^{f/2}( \partial_{i}e^{f})e^{f/2} )^{*}) \leq 0.
\end{split}
\end{equation*}

Comparing this with Proposition \ref{E-Himproved} can reveal the possibility of existing more links between noncommutative connection theory and noncommutative spectral theoretical approaches.
\end{remark}

The non-conformally rescaled metrics have rarely been explored, especially in higher dimensions, cf. \cite{Dab1}. In the following proposition we prove the non-positivity of the E-H action for some other non-conformal diagonal metrics.   

\begin{proposition} \label{E-Hnewprop}
Given all the assumptions of Proposition \ref{E-Himproved}, except that we consider the metric 
$$ \begin{pmatrix}
e^{f} & 0 & 0 & 0 \\
0 & e^{f} & 0 & 0 \\
0 & 0 & 1 & 0 \\
0 & 0 & 0 & 1 \\
\end{pmatrix}.$$ 
Then the associated Einstein-Hilbert action functional is non-positive.
\end{proposition}
\begin{proof}
We first compute the the connection coefficients
$$\bigtriangledown_{1} \partial_{1} = \frac{1}{2} \big\{ (\partial_{1}e^{f})e^{-f} \partial_{1} - (\partial_{2}e^{f})e^{-f} \partial_{2}- (\partial_{3}e^{f}) \partial_{3}-(\partial_{4}e^{f}) \partial_{4}   \big\},$$
$$\bigtriangledown_{2} \partial_{2} = \frac{1}{2} \big\{- (\partial_{1}e^{f})e^{-f} \partial_{1} + (\partial_{2}e^{f})e^{-f} \partial_{2}- (\partial_{3}e^{f}) \partial_{3}-(\partial_{4}e^{f}) \partial_{4}   \big\},$$ 
$$\bigtriangledown_{1} \partial_{2}=\bigtriangledown_{2} \partial_{1} = \frac{1}{2} \big\{ (\partial_{2}e^{f})e^{-f} \partial_{1} + (\partial_{1}e^{f})e^{-f} \partial_{2}  \big\},$$

$$ \bigtriangledown_{3} \partial_{4}=\bigtriangledown_{4} \partial_{3}=0, $$

$$\bigtriangledown_{1} \partial_{3}=\frac{1}{2} (\partial_{3}e^{f}) e^{-f} \partial_{1}, \hspace{1 cm} 
\bigtriangledown_{1} \partial_{4}=\frac{1}{2} (\partial_{4}e^{f}) e^{-f} \partial_{1},  $$

$$\bigtriangledown_{2} \partial_{3}= \frac{1}{2} (\partial_{3}e^{f}) e^{-f} \partial_{2}, \hspace{1 cm}  
\bigtriangledown_{2} \partial_{4}= \frac{1}{2} (\partial_{4}e^{f}) e^{-f} \partial_{2}. $$
To compute the scalar curvature, we see that 
\begin{equation*} 
\begin{split}
R & = e^{-2f} \big( R_{1212} + R_{2121} \big)+ e^{-f} \big( R_{1313} + R_{3131} + R_{1414}  \\
 & + R_{4141} + R_{2323} + R_{3232} + R_{2424} + R_{4242} \big) + R_{3434} + R_{4343}.
\end{split}
\end{equation*}
Due to the symmetries between the indices $1,2$ and $3,4$, we will have 
$$R =2 e^{-2f}R_{1212} + 2e^{-f} \big( R_{1313} + R_{1414}+ R_{2323} + R_{2424} \big) + 2 R_{3434}  $$
where 

\begin{equation*}
\begin{split}
R_{1212} & = - \frac{1}{2} \partial_{1} \big( (\partial_{1}e^{f})e^{-f}\big) e^{f} - \frac{1}{2}  \partial_{2}\big( (\partial_{2}e^{f})e^{-f}\big) e^{f} \\
 & - \frac{1}{4} (\partial_{3}e^{f}) (\partial_{3}e^{f}) - \frac{1}{4} (\partial_{4}e^{f})(\partial_{4}e^{f}),\\
 R_{1313} & = - \frac{1}{2} \partial_{3}^{2} e^{f} + \frac{1}{4} (\partial_{3}e^{f}) e^{-f} (\partial_{3}e^{f}), \hspace{1 cm}\\
 R_{1414} & = - \frac{1}{2} \partial_{4}^{2} e^{f} + \frac{1}{4} (\partial_{4}e^{f}) e^{-f} (\partial_{4}e^{f}),\\
 R_{3131} & = - \frac{1}{2} \partial_{3} \big( (\partial_{3}e^{f}) e^{-f} \big) e^{f} - \frac{1}{4} (\partial_{3}e^{f}) e^{-f} (\partial_{3}e^{f}), \hspace{1 cm} \\
 R_{4141} & = - \frac{1}{2} \partial_{4} \big( (\partial_{4}e^{f}) e^{-f} \big) e^{f} - \frac{1}{4} (\partial_{4}e^{f}) e^{-f} (\partial_{4}e^{f}),
\end{split}
\end{equation*}
and $R_{3434}=0$. The E-H action is then computed as following
\begin{equation*}
\begin{split}
E-H :=&  \tau(R \sqrt{g})= \tau(R e^{f})= 2 \tau \big( e^{-f} R_{1212} + R_{1414} + R_{1313} + R_{3131} + R_{4141} \big) \\
= &  2 \tau \Big(   
- \frac{1}{2} \partial_{1} \big( (\partial_{1}e^{f})e^{-f}\big) - \frac{1}{2}  \partial_{2}\big( (\partial_{2}e^{f})e^{-f}\big) - \frac{1}{4} e^{-f}(\partial_{3}e^{f}) (\partial_{3}e^{f}) - \frac{1}{4} e^{-f}(\partial_{4}e^{f})(\partial_{4}e^{f})
 \Big)\\
& + 2 \tau \Big( - \frac{1}{2} \partial_{3}^{2} e^{f}  - \frac{1}{2} \partial_{4}^{2} e^{f} - \frac{1}{2} \partial_{3} \big( (\partial_{3}e^{f}) e^{-f} \big) e^{f} - \frac{1}{2} \partial_{4} \big( (\partial_{4}e^{f}) e^{-f} \big) e^{f} \Big)\\
= &  2 \tau \Big( 
 - \frac{1}{4} e^{-f}(\partial_{3}e^{f}) (\partial_{3}e^{f}) - \frac{1}{4} e^{-f}(\partial_{4}e^{f})(\partial_{4}e^{f})  + \frac{1}{2} (\partial_{3}e^{f}) e^{-f} (\partial_{3} e^{f})+ \frac{1}{2}(\partial_{4}e^{f}) e^{-f}(\partial_{4} e^{f}) \Big)
\end{split}
\end{equation*}
\begin{equation} \label{traceE-Hnewprop}
\begin{split}
& =2 \tau \Big(
+ \frac{1}{4} (\partial_{3}e^{f}) e^{-f} (\partial_{3} e^{f})+ \frac{1}{4}(\partial_{4}e^{f}) e^{-f}(\partial_{4} e^{f}) \Big) \\
& = \frac{1}{2} \tau \Big((\partial_{3}e^{f}) e^{-f} (\partial_{3} e^{f}) + (\partial_{4}e^{f}) e^{-f}(\partial_{4} e^{f})  \Big)\\
& = \frac{1}{2} \tau \Big((\partial_{3}e^{f}) e^{-\frac{f}{2}} e^{-\frac{f}{2}} (\partial_{3} e^{f}) + (\partial_{4}e^{f}) e^{-\frac{f}{2}} e^{-\frac{f}{2}}(\partial_{4} e^{f}) \Big)\\
& - \frac{1}{2} \tau \big( (e^{-\frac{f}{2}}(\partial_{3}e^{f}))^{*} e^{-\frac{f}{2}} (\partial_{3} e^{f}) + (e^{-\frac{f}{2}}(\partial_{4}e^{f}))^{*}  e^{-\frac{f}{2}}(\partial_{4} e^{f})\big) \leq 0.
\end{split}
\end{equation} 
\end{proof}
In spite of Proposition \ref{E-Himproved}, we do not have the extremum property in Proposition \ref{E-Hnewprop} since $\partial_{1}$ and $\partial_{2}$ are no longer present in the curvature formula (\ref{traceE-Hnewprop}). Besides, due to the symmetry of the metric, the same result in Proposition \ref{E-Hnewprop} holds for metrics such as 
$$ \begin{pmatrix}
e^{f} & 0 & 0 & 0 \\
0 & 1 &  0 & 0 \\
0 & 0 & e^{f} & 0 \\
0 & 0 & 0  & 1 \\
\end{pmatrix}. $$
\\
%...................................................................%
%...................................................................%

\subsection{The gradient of Einstein-Hilbert action}
An important step toward the associated flow of the E-H action, is computing the gradient of the E-H action. The gradient of this action was discussed and obtained in \cite{Fat1} in the spectral setting by means of rearrangement Lemma \cite{Con2}.  We here obtain an analogue of that result for the E-H action in our setting.    
\\
Consider the E-H action as a function of the self-adjoint element $f \in A_{\theta}^{\infty}$
\begin{equation} \label{E-Hrelation}
\Omega(f)= \frac{3}{2} \Sigma_{i =1}^{4} \tau \big( (\partial_{i} e^{f}) e^{-f} (\partial_{i} e^{f}) \big),
\end{equation}
To compute the gradient of E-H action, we should compute $\frac{d}{dt}_{|t=0} \Omega(f+th) $ where $h = h^{*} \in A_{\theta}^{\infty}$. As $\frac{d}{dt}$ commutes with $\partial_{i}$ and the Leibniz rule holds for $\frac{d}{dt}$, we get

\begin{equation*}
\begin{split}
\frac{d}{dt} \Omega(f+th) = &  \frac{3}{2} \Sigma_{i =1}^{4} \tau \big \{   
 (\partial_{i} \frac{d}{dt} e^{f+th}) e^{-f-th} (\partial_{i} e^{f+th})\\
 & + (\partial_{i} e^{f+th}) (\frac{d}{dt}e^{-f-th}) (\partial_{i} e^{f+th})  + (\partial_{i} e^{f+th}) e^{-f-th} (\partial_{i} \frac{d}{dt}e^{f+th}) \big\}.
\end{split}
\end{equation*}
Using the following assertion from \cite{Con2} 
$$\frac{d}{dt} e^{f+th} =\frac{1-e^{\bigtriangledown}}{\bigtriangledown}(-h) e^{f}=- \frac{1-e^{\bigtriangledown}}{\bigtriangledown}(h) e^{f},$$   
we will have

\begin{equation} \label{GradientE-H}
\begin{split}
& \frac{d}{dt}_{|t=0} \Omega(f+th)= \frac{3}{2} \Sigma_{i =1}^{4} \tau \big \{- (\partial_{i} (\frac{1-e^{\bigtriangledown}}{\bigtriangledown}(h) e^{f}) e^{-f} (\partial_{i} e^{f}) \\
&\hspace{1 cm} + (\partial_{i} e^{f}) (\frac{1- e^{\bigtriangledown}}{\bigtriangledown}(h) e^{-f}) (\partial_{i} e^{f}) - (\partial_{i} e^{f}) e^{-f} (\partial_{i}(1-\frac{e^{-\bigtriangledown}}{\bigtriangledown}(h) e^{f})) \big\},
\end{split}
\end{equation}
which by using the tracial property of $\tau$ and integration by parts is equal to

\begin{equation*}
\begin{split}
 = & \frac{3}{2} \Sigma_{i =1}^{4} \tau \big \{\frac{1- e^{\bigtriangledown}}{\bigtriangledown}(h) e^{f} \partial_{i}(e^{-f} (\partial_{i} e^{f})) + \frac{1- e^{\bigtriangledown}}{\bigtriangledown}(h) e^{-f} (\partial_{i} e^{f})(\partial_{i} e^{f}) \\
& + \frac{1 - e^{\bigtriangledown}}{\bigtriangledown}(h) e^{f} \partial_{i}((\partial_{i} e^{f}) e^{-f}) \big\},
\end{split}
\end{equation*}
\\
which by applying the formula $\tau(\bigtriangledown(h)e^{f}x) = \tau(h e^{f} (-\bigtriangledown(x)))$ becomes

\begin{equation} \label{gradientder}
\begin{split}
= & \frac{3}{2} \Sigma_{i =1}^{4} \tau \big \{h e^{f} \frac{e^{- \bigtriangledown}-1}{\bigtriangledown}\partial_{i}(e^{-f} (\partial_{i} e^{f})) \\
& + h e^{-f} \frac{e^{- \bigtriangledown}-1}{\bigtriangledown}(\partial_{i} e^{f})(\partial_{i} e^{f})
+ h e^{f} \frac{e^{- \bigtriangledown}-1}{\bigtriangledown} \partial_{i}((\partial_{i} e^{f}) e^{-f}) \big\}.
\end{split}
\end{equation}

Now since by definition we have 
\begin{equation} \label{gradientderat0}
\frac{d}{dt}|_{t=0} \Omega(f+th) = <Grad_{f} \Omega, h>= \tau(h Grad_{f} \Omega).
\end{equation}

The relations (\ref{gradientder}) and (\ref{gradientderat0}) induce

\begin{equation} \label{gradientE-Hfinal}
\begin{split}
Grad_{f} \Omega = & \frac{3}{2} \Sigma_{i =1}^{4} \tau \big \{ e^{f} \frac{e^{- \bigtriangledown}-1}{\bigtriangledown}\partial_{i}(e^{-f} (\partial_{i} e^{f})) \\
 & + e^{-f} \frac{e^{- \bigtriangledown}-1}{\bigtriangledown}(\partial_{i} e^{f})(\partial_{i} e^{f})
+ e^{f} \frac{e^{- \bigtriangledown}-1}{\bigtriangledown} \partial_{i}((\partial_{i} e^{f}) e^{-f}) \big\}.
\end{split}
\end{equation}

So we have shown
\begin{proposition}
The gradient of the E-H action of the conformal metric is given by (\ref{gradientE-Hfinal}).
\end{proposition}

%...................................................................%
%...................................................................%

\subsection{Computing the scalar curvature in terms of $\delta_{i}(h)$ and $\nabla$}
We find in this part the curvature in (\ref{EHRfinal}) in terms of $\nabla = log \Delta := - ad_{f}$ and $\delta_{i}(f)$. The idea is based on the work of \cite[Eq. (1)]{Fat1}, in which the curvature 
\begin{equation} \label{Farzadcurvature}
R = \pi^{2} \Sigma_{i=1}^{4} \big( - e^{-h} \partial_{i}^{2}(e^{h})e^{-h} + \frac{3}{2} e^{-h} \partial_{i}(e^{h})e^{-h} \partial_{i}(e^{h}) e^{-h} \big),
\end{equation}
 is shown up to a constant to be equal to
\begin{equation} 
R = e^{-f} K(\nabla)(\Sigma_{i=1}^{4} \delta_{i}^{2}f)+ e^{-f} H(\nabla,\nabla)(\Sigma_{i=1}^{4} \delta_{i}f^{2}),
\end{equation}
where $K$ and $H$ are some analytic functions given in (\ref{K}) and (\ref{H}). On the other hand, the scalar curvature  in relation (\ref{EHRfinal}) is given by
\begin{equation} \label{ourcurvature}
R =\Sigma_{i=1}^{4} -3 e^{-f} e^{-f}(\partial_{i}^{2} e^{f}) + \frac{3}{2} e^{-2f}(\partial_{i} e^{f}) e^{-f} (\partial_{i} e^{f}).
\end{equation}
Now consider the identities
\begin{equation} \label{identities}
e^{-f} \delta_{i}(e^{f}) = g_{1}(\Delta)(\delta_{i}f), \hspace{1 cm} e^{-f} \delta_{i}^{2}(e^{f}) = g_{1}(\Delta)(\delta_{i}^{2}f) + 2 g_{2}(\Delta , \Delta) (\delta_{i}(f) \delta_{i}(f)),
\end{equation}
into which 
$$g_{1} (\mu)= \frac{\mu-1}{log \mu}, \hspace{1 cm} g_{2}(\mu, \nu)= \frac{\mu(\nu -1)log \mu - (\mu -1) log \nu}{log \mu log \nu (log \mu) + log \nu) }.$$
By substituting the relations (\ref{identities}) into (\ref{ourcurvature}) we obtain

\begin{equation*}
\begin{split}
R & = \Sigma_{i=1}^{4}\big\{ -3 e^{-f} g_{1}(\Delta)(\delta_{i}^{2}f) - 6  e^{-f}g_{2}(\Delta , \Delta) (\delta_{i}(f)^{2}) + \frac{3}{2} e^{-f}g_{1}(\Delta)(\delta_{i}f)g_{1}(\Delta)(\delta_{i}f) \big\}\\
& = -3 e^{-f} K(\nabla)(\Sigma_{i=1}^{4} \delta_{i}^{2}f) +3 e^{-f} H(\nabla, \nabla)(\Sigma_{i=1}^{4} \delta_{i}(f)^{2}),
\end{split}
\end{equation*}
where 
 \begin{equation} \label{K}
K(s) = \frac{e^{s} -1}{s}, 
\end{equation}
and
\begin{equation} \label{H}
\begin{split}
H(s, t)& = -2 g_{2}(e^{s}, e^{t})+ \frac{1}{2} g_{1}(e^{s}) g_{1}(e^{t}) \\
& = \frac{-2se^{s}(e^{t}-1)+2t (e^{s}-1)}{st(s+t)} +\frac{(e^{s}-1)(e^{t}-1)}{2st}\\
& =\frac{s (e^{t}-1)(-3 e^{s} -1) + t (e^{s}-1)(e^{t} +3)}{2st (s+t)} .
\end{split}
\end{equation}
We summarize the above result in the following proposition.

\begin{proposition} \label{RNablaDelta}
The curvature in \ref{ourcurvature} for noncommutative four-torus with the conformal flat metric given by (\ref{ourcurvature}) has the following form:
\begin{equation} \label{RNablaDelta2}
R= -3 e^{-f} K(\nabla)(\Sigma_{i=1}^{4} \delta_{i}^{2}f) +3 e^{-f} H(\nabla, \nabla)(\Sigma_{i=1}^{4} \delta_{i}(f)^{2}),
\end{equation}
where $K$ and $H$ are given in (\ref{K}) and (\ref{H}).
\end{proposition}

\subsection{Projections and the scalar curvature}
We plan in this subsection to obtain the scalar curvature in terms of dilatons $f = sp$, where $s \in \mathbb{R}$ and $p= p^{*} = p^{2}$ is a projection. We follow the similar lines in \cite{Con2} and \cite[Section IV]{Fat1}. Consider the conformal flat metric on noncommutative $4-$torus, $g_{ij} = e^{f} \delta_{ij}$, $i,j \in \{1,...,4\}$. The scalar curvature given in (\ref{EHRfinal}) is
\begin{equation*}
R =\Sigma_{i=1}^{4} -3 e^{-2f} (\partial_{i}^{2} e^{f}) + \frac{3}{2} e^{-2f}(\partial_{i} e^{f}) e^{-f} (\partial_{i} e^{f}). 
\end{equation*}
We show that the terms in the right hand side of the identity 
\begin{equation} \label{laplacianidentity}
\bigtriangleup (p) = p \bigtriangleup (p) p + p \bigtriangleup (p) (1-p) +  (1-p) \bigtriangleup (p) p + (1-p) \bigtriangleup (p) (1-p) 
\end{equation} 
are the eigenvalues of the operator $\nabla = -ad_{f}$ corresponding respectively to the eigenvalues $0,s,-s$ and $0$. To see this, we use the identity 
$$\delta_{j}(p) = p \delta_{j}(p) + \delta_{j}(p) p.$$

We will have
$$\nabla (p \bigtriangleup (p) p) = -s\big\{p^{2} \bigtriangleup (p) p-p \bigtriangleup (p) p^{2}   \big\} = 0,$$
so

\begin{equation} \label{eq1}
\begin{split}
\nabla (p \bigtriangleup (p) (1-p)) & =\nabla (p \bigtriangleup (p)) = -s \big\{ p^{2} \bigtriangleup (p)-p \bigtriangleup (p)p   \big\} \\
 & = -s p\bigtriangleup (p) (1-p),
\end{split}
\end{equation}

and similarly for other two cases. So applying the identity \label{laplacianidentity}, it yields

\begin{equation} \label{K(Nabla)}
\begin{split}
K(\nabla)( \bigtriangleup (f)) = & sK(\nabla)( \bigtriangleup (p)) \\
= &  s K(0) (p \bigtriangleup (p) p + (1-p) \bigtriangleup (p) (1-p) )\\
& + sK(-s)(p \bigtriangleup (p) (1-p)) + sK(s)((1-p) \bigtriangleup (p) p)\\
= & s (p \bigtriangleup (p) p + (1-p) \bigtriangleup (p) (1-p) )\\
& +(1-e^{-s})(p \bigtriangleup (p) (1-p)) + (e^{s}-1)((1-p) \bigtriangleup (p) p).\\
\end{split}
\end{equation}

To find a final formula for the scalar curvature, we also need to find similar decomposition for $H(\nabla,\nabla)$. According to \cite[Proposition 4.1]{Fat1} we have the relation

\begin{equation} \label{Hdecomposition}
\begin{split}
& H(\bigtriangledown, \bigtriangledown)(\delta_{i}(f)\delta_{i}(f)) \\
& = \frac{s^{2}}{2} \Big( (H(s,-s)+H(-s,s)) + (H(-s,-s)- H(s,-s))(1-2p)\Big) (\delta_{i}(p)\delta_{i}(p)). 
\end{split}
\end{equation}
The term $H(s,t)$ is denoted in \cite{Fat1} by $-\tilde{H}(t,s)$. By \cite[Theorem 3.1]{Fat1} we have
\begin{equation}
H(s,t)= -2 \frac{K(s+t)-K(t)}{s} - \frac{3}{2} K(t) K(s).
\end{equation} 

So we obtain
\begin{equation} \label{H's}
\begin{split}
H(s,-s)+H(-s,s) & = -\frac{5}{s^{2}} (e^{s}+ e^{-s} -2) ,\\
H(-s,s)- H(s,-s) & = \frac{4s-4sinh s}{s^{2}}.
\end{split}
\end{equation}
which gives
\begin{equation} \label{Hdecomposition2}
\begin{split}
& H(\bigtriangledown, \bigtriangledown)(\delta_{i}(f)\delta_{i}(f)) \\
& = \frac{s^{2}}{2} \Big( \frac{10 (1- cosh s)}{s^{2}} + (\frac{4s-4sinh s}{s^{2}})(1-2p)\Big) (\delta_{i}(p)\delta_{i}(p)) \\
& = \Big( 5 (1- cosh s) + (2 s-2sinh s)(1-2p)\Big) (\delta_{i}(p)\delta_{i}(p)) \\
& = \Big(5 -5coshs+ 2s - 2sinhs + (4s-4sinhs)p \Big) (\delta_{i}(p)\delta_{i}(p)). 
\end{split}
\end{equation}
\\
We also have the identity
\begin{equation} \label{deltasquare}
\Sigma \delta_{i}(p)^{2} = \frac{1}{2} \Big( (1-p)\bigtriangleup(p) -  \bigtriangleup(p)p \Big),
\end{equation}
that substituting into (\ref{Hdecomposition2}) gives
\begin{equation} \label{Hdecomposition3}
\begin{split}
& \frac{1}{2} H(\bigtriangledown, \bigtriangledown)((1-p)\bigtriangleup(p) -  \bigtriangleup(p)p) \\
& =(\frac{5}{2} -\frac{5}{2} \cosh s+ s - \sinh s) \bigtriangleup(p) \\
&- (\frac{5}{2} -\frac{5}{2} coshs+ s - sinhs) p \bigtriangleup(p) \\
&- (\frac{5}{2} -\frac{5}{2} coshs+ s - sinhs) \bigtriangleup(p)p \\
&-(2s-2sinhs)p \bigtriangleup(p)p .
\end{split}
\end{equation}
Now by substituting the relations (\ref{K(Nabla)}) and (\ref{Hdecomposition3}) into the relation  (\ref{RNablaDelta2}) we obtain the following expression for the curvature for the dilaton $f= sp$
\begin{equation}  \label{Rdilaton}
\begin{split}
\frac{1}{3}e^{sp} R = & -s \bigtriangleup(p)+ (-1+s+e^{-s}) p\bigtriangleup(p) +(1+s- e^{s})  \bigtriangleup(p) p \\
& +(-2s+2sinhs)p \bigtriangleup(p) p  +(\frac{5}{2} -\frac{5}{2} coshs+ s - sinhs) \bigtriangleup(p) \\
& - (\frac{5}{2} -\frac{5}{2} coshs+ s - sinhs) p \bigtriangleup(p) \\
&- (\frac{5}{2} -\frac{5}{2} coshs+ s - sinhs) \bigtriangleup(p)p 
-(2s-2sinhs)p \bigtriangleup(p)p \\
= &(\frac{5}{2} -\frac{5}{2} coshs - sinhs) \bigtriangleup(p)
+ (-\frac{7}{2} + e^{-s}+\frac{5}{2} coshs + sinhs)  p \bigtriangleup(p)  \\
& + (-\frac{3}{2} - e^{s}+\frac{5}{2} coshs + sinhs ) \bigtriangleup(p) p   + (-4s+4sinhs)  p \bigtriangleup(p) p.
\end{split}
\end{equation}

%...........................................................
%...........................................................

\subsection{New proof of the non-positivity of E-H action}
It is shown in \cite [p. 6]{Fat1} that up to a constant the curvature %(\ref{Farzadcurvature}) 
\begin{equation} \label{Farzadcurvature}
R = \pi^{2} \Sigma_{i=1}^{4} \big( - e^{-h} \partial_{i}^{2}(e^{h})e^{-h} + \frac{3}{2} e^{-h} \partial_{i}(e^{h})e^{-h} \partial_{i}(e^{h}) e^{-h} \big)
\end{equation}
is equal to the curvature obtained by Fathizadeh and Khalkhali in \cite{Mas3}. They have shown in \cite[Section 5]{Mas3}, that $\tau(R) \leq 0$ and that the extremal case occurs for a constant multiple of the flat metric. They applied definitions $\nabla$ and $\Delta$ and used the positivity of some analytic functions in their proof. We prove in the following their result by the method we used in our proof for the nonpositivity of E-H action in \ref{E-Himproved}.

\begin{proposition} \label{MasoudFarzadEH}
Consider the spectral triple on noncommutative torus given in \cite{Mas3} and \cite{Fat1}. Consider also its associated scalar curvature given in the formula (\ref{Farzadcurvature}). Then the Einstein-Hilbert action functional for the scalar curvature, $\tau(R)$, is non-positive and attains its extremum for a constant multiple of the flat metric.
\end{proposition}

\begin{proof}
Consider the the trace for the scalar curvature in (\ref{Farzadcurvature})

\begin{equation} \label{MFEHtrace}
\begin{split}
\frac{1}{\pi^{2}} \tau(R) & = \Sigma_{i=1}^{4} \tau \big(- e^{-h} \delta_{i}^{2}(e^{h})e^{-h} + \frac{3}{2} e^{-h} \delta_{i}(e^{h})e^{-h} \delta_{i}(e^{h}) e^{-h} \big) \\
 & = \Sigma_{i=1}^{4}\big\{ -\tau \big( e^{-2h} \delta_{i}^{2}(e^{h}) \big) +\frac{3}{2} \tau\big( e^{-2h} \delta_{i}(e^{h})e^{-h} \delta_{i}(e^{h})\big) \big\} ,
\end{split}
\end{equation}

where for the first term in (\ref{MFEHtrace}) we have 

\begin{equation} \label{MFsimplifiedterm}
\begin{split}
\tau \big( e^{-2h} \delta_{i}^{2}(e^{h}) \big) & = - \tau \big( \delta_{i} (e^{-2h}) \delta_{i}(e^{h}) \big) \\
& =  - \tau \big(e^{-h} \delta_{i} (e^{-h}) \delta_{i}(e^{h}) \big) - \tau \big( \delta_{i} (e^{-h}) e^{-h} \delta_{i}(e^{h}) \big) \\
& = -\tau \big(e^{-2h} \delta_{i} (e^{h}) e^{-h} \delta_{i}(e^{h}) \big) - \tau \big(e^{-h} \delta_{i} (e^{h}) e^{-2h} \delta_{i}(e^{h}) \big)\\
& = -2\tau \big(e^{-2h} \delta_{i} (e^{h}) e^{-h} \delta_{i}(e^{h}) \big)
\end{split}
\end{equation}

Substituting (\ref{MFsimplifiedterm}) into (\ref{MFEHtrace}) gives

\begin{equation*} 
\begin{split}
\frac{1}{\pi^{2}} \tau(R) & =\Sigma_{i=1}^{4}\big\{ 2  \tau \big(e^{-2h} \delta_{i} (e^{h}) e^{-h} \delta_{i}(e^{h}) \big) + \frac{3}{2} \tau \big(e^{-2h} \delta_{i} (e^{h}) e^{-h} \delta_{i}(e^{h}) \big)  \big\}\\
& = \frac{7}{2}\Sigma_{i=1}^{4} \tau \big(e^{-2h} \delta_{i} (e^{h}) e^{-h} \delta_{i}(e^{h}) \big) \\
& = - \frac{7}{2}\Sigma_{i=1}^{4} \tau \big(e^{-2h} \big( \delta_{i} (e^{h}) e^{- \frac{h}{2}} \big) \big( \delta_{i}(e^{h}) e^{- \frac{h}{2}} \big)^{*} \big)  \leq 0.
\end{split}
\end{equation*}

And the extremum case occurs only when 

\begin{equation} \label{eq1}
\begin{split}
& e^{-2h} \big(\delta_{i} (e^{h}) e^{- \frac{h}{2}} \big) \big( \delta_{i}(e^{h}) e^{- \frac{h}{2}} \big)^{*}= 0, \\
 & \delta_{i}(e^{h}) e^{- \frac{h}{2}}=0,\\
 & \delta_{i}(e^{h})=0, i \in \{1,...,4 \},
\end{split}
\end{equation} 
so $e^{h}$ is constant and this proves the proposition.
\end{proof}

%.................................................................
%.................................................................

\section{On the Gauss-Bonnet theorem} \label{SecGB}
 
Consider the conformal flat metric 
\begin{equation} \label{conformametrictwodim}
\begin{pmatrix}
e^{f} & 0 \\
0 & e^{f} \\
\end{pmatrix}
\end{equation}
for a self-adjoint element $f \in A_{\theta}^{\infty}$. The scalar curvature computed in \cite{Ros1} is
\begin{equation*} \label{NCR1212}
R_{1212}= - \frac{1}{2}  \big\{ \partial_{2}(\partial_{2}(e^{f}) e^{-f})  + \partial_{1}(\partial_{1}(e^{f}) e^{-f}) \big\} e^{f},
\end{equation*}
so
\begin{equation} \label{NCR}
R = 2e^{-2f}R_{1212}=  - e^{-2f}  \big\{ \partial_{2}(\partial_{2}(e^{f}) e^{-f})  + \partial_{1}(\partial_{1}(e^{f}) e^{-f}) \big\} e^{f}.
\end{equation}
From Proposition \ref{G-B} we have the Gauss-Bonnet theorem, $\tau(R \sqrt{detg})=0$.
\\
%...................................................................%
%...................................................................%

The metrics that have been usually worked on were conformally  flat, \cite{Con3, Mas1}. Now consider the non-conformal flat metric
$ \begin{pmatrix}
e^{f} & 0 \\
0 & 1 \\
\end{pmatrix} $ 
which is similar to the metric introduced in \cite{Dab1}. Its scalar curvature by definition is 

\begin{equation} \label{eq1}
\begin{split}
R = g^{11} R_{11} + g^{22}R_{22} & = g^{11}g^{22} R_{1212} + g^{22}g^{11}R_{2121} \\
& = e^{-f}(R_{1212} + R_{2121}),
\end{split}
\end{equation}

where we have 
$$R_{1212} = - \frac{1}{2} \partial_{2}^{2} e^{f} + \frac{1}{4} \partial_{2}e^{f} e^{-f} \partial_{2}e^{f},$$
and 
$$R_{2121} = - \frac{1}{2} \partial_{2} \big( (\partial_{2}e^{f}) e^{-f}\big) e^{f} - \frac{1}{4} \partial_{2}e^{f} e^{-f} \partial_{2}e^{f}.$$
So
\begin{equation} \label{scalarcurvature}
R =  e^{-f} \big\{ - \frac{1}{2} \partial_{2}^{2} e^{f}  - \frac{1}{2} \partial_{2} \big( (\partial_{2}e^{f}) e^{-f}\big) e^{f} \big\}.
\end{equation}

This implies
\begin{equation} \label{GBnondiagonal}
\begin{split}
\tau (R \sqrt{g}) & = \tau(R e^{\frac{f}{2}})  \\
 & = - \frac{1}{2} \tau \big( e^{-\frac{f}{2}}    \partial_{2}^{2} e^{f} +e^{\frac{f}{2}}  \partial_{2} \big( (\partial_{2}e^{f}) e^{-f}\big) \big)\\
 & =\frac{1}{2} \tau\big( (\partial_{2}e^{-\frac{f}{2}}) (\partial_{2}e^{f}) + (\partial_{2}e^{\frac{f}{2}})(\partial_{2}e^{f}) e^{-f} \big).
\end{split}
\end{equation}

\begin{remark} \label{nonsymmetrictensorremark}
In relation (\ref{scalarcurvature}), we clearly see that in spite of the classical case we do not have the equality $R_{1212} = R_{2121}$. So we note that we do not have the anti-symmetry property $R_{ijkl}= -R_{ijlk}$ in general, cf. \cite[Proposition 3.4]{Ros1}.
\end{remark}

%...................................................................%
%...................................................................%
\subsection{Commutativity of metric element with partial derivatives}

We obtained above the scalar curvature in (\ref{scalarcurvature}) for the metric
$ \begin{pmatrix}
e^{f} & 0 \\
0 & 1 \\
\end{pmatrix} $
as
$$R =  e^{-f} \big\{ - \frac{1}{2} \partial_{2}^{2} e^{f}  - \frac{1}{2} \partial_{2} \big( (\partial_{2}e^{f}) e^{-f}\big) e^{f} \big\}.$$
To check the Gauss-Bonnet theorem for this metric, we add an additional assumption of commutativity of $f$ with its partial derivatives $\partial_{1}$ and $\partial_{2}$ and so the relation $\partial_{2}^{2}e^{f}= \big(  \partial_{2}^{2}f + (\partial_{2}f)^{2} \big) e^{f}$ to reach to 
\begin{equation} \label{scalarcommutativity}
R = \big\{ -  \partial_{2}^{2}f - \frac{1}{2}(\partial_{2}f)^{2} \big\}.
\end{equation}

The scalar curvature in (\ref{scalarcommutativity}) coincides with the classical case when $\theta = 0$,(it can be computed by hand or checked by Mathematica package). We choose in continue a specific self-adjoint element in $A_{\theta}^{\infty}$ and use its Fourier expansion to reach to the same series as in the classical case where the Riemannian integral of the curvature vanishes. According to the relation (\ref{scalarcommutativity} we have
\begin{equation} \label{G-Btau}
\tau (R \sqrt{detg}) = \tau  \big(  e^{\frac{f}{2}}  \big\{ -  \partial_{2}^{2}f - \frac{1}{2} (\partial_{2}f)^{2} \big\}\big).
\end{equation}
\\
We show in the following the relation (\ref{G-Btau}) for the self-adjoint element $f=2( V+V^{-1})$ vanishes. 
\begin{proposition} \label{G-Bsecondmetric}
Consider the self-adjoint element $f= 2(V+V^{-1})$ in the noncommutative two-torus $A_{\theta}$ and the metric 
$\begin{pmatrix}
e^{f} & 0 \\
0 & 1 \\
\end{pmatrix}$
on the vector bundle $X_{\theta}$. The associated scalar curvature (\ref{scalarcommutativity}) satisfies in the Gauss-Bonnet theorem $\tau(R \sqrt{detg}) = \tau(R e^{f/2})=0$.
\end{proposition}
\begin{proof}
We have
\begin{equation} \label{examplecurvature}
\frac{1}{2} R \sqrt{detg} =  e^{V+V^{-1}}\big(  2 - V- V^{-1} - V^{2} -V^{-2}   \big).
\end{equation}
So to find the trace of the curvature in relation (\ref{examplecurvature}), we should find the coefficients of $V$, $V^{-1}$ , $V^{2}$, $V^{-2}$ and the constant term of the exponential $e^{V+V^{-1}}$. We need to find the expansion of the exponential term
\begin{equation} \label{exp}
 e^{V+V^{-1}} = \Sigma_{k=0}^{\infty} \frac{(V+V^{-1})^{q}}{q!}.
\end{equation} 
Since $V$ and $V^{-1}$ commute with each other, we have the binomial expansion
\begin{equation} \label{binomial}
 (V+V^{-1})^{q} = \Sigma_{k=0}^{q}V^{k} V^{k-q}\left( \begin{array}{ccc} 
q \\
k
\end{array} \right).
\end{equation}
To get the coefficient of $V$ in relations (\ref{exp}) and (\ref{binomial}), put 
$$2k-q=1 \rightarrow k= \frac{q+1}{2}.$$
So for each integer $\frac{q+1}{2}$ we get the coefficient of $V$ in binomial expansion as
 $\frac{ \binom{q}{\frac{q+1}{2}}}{q!}$
that sums all to 
$$\Sigma_{\frac{q+1}{2} integer}\frac{ \binom{q}{\frac{q+1}{2} } }{(q)!},$$    
and the same for $V^{-1}$. By a change of variable, we obtain the coefficients of $V$ and $V^{-1}$ as
$$\Sigma_{q=0}^{\infty}\frac{\binom{2q+1}{q}}{(2q+1)!}.$$
\\
Doing the same combinatorics give the coefficients of $V^{2}$ and $V^{-2}$ as 
 $$\Sigma_{q=0}^{\infty}\frac{ \left( \begin{array}{ccc} 
2q+2 \\
q
\end{array} \right)}{(2q+2)!},$$
 and the constant term of the exponential is
 
 \begin{equation*}
 \Sigma_{q=0}^{\infty}\frac{ \left( \begin{array}{ccc} 
2q \\
q
\end{array} \right)}{(2q)!}.
\end{equation*}
So the trace in relations (\ref{G-Btau}) and (\ref{examplecurvature}) becomes

\begin{equation*}
\begin{split}
\frac{1}{2} \tau(R \sqrt{det g}) & = 2 \bigg( \Sigma_{q=0}^{\infty}\frac{ \left( \begin{array}{ccc} 
2q \\
q
\end{array} \right)}{(2q)!} 
-\Sigma_{q=0}^{\infty}\frac{ \left( \begin{array}{ccc} 
2q+1 \\
q
\end{array} \right)}{(2q+1)!}
-\Sigma_{q=0}^{\infty}\frac{ \left( \begin{array}{ccc} 
2q+2 \\
q
\end{array} \right)}{(2q+2)!}
\bigg) \\
& =2 \bigg( \Sigma_{q=0}^{\infty}\frac{ 1}{q! q!}
  -\Sigma_{q=0}^{\infty}\frac{ 1}{q! (q+1)!}
 -\Sigma_{q=0}^{\infty}\frac{ 1}{q! (q+2)!} \bigg),
\end{split}
\end{equation*}
and by telescoping sum
\begin{equation} \label{finalseries}
= \Sigma_{q=0}^{\infty}\frac{ q^{2} +2q -1}{q! (q+2)!}.
\end{equation}
\\
Separating the series (\ref{finalseries}) into two parts gives

\begin{equation*} 
\begin{split}
\Sigma_{q=0}^{\infty}\frac{ q^{2} +2q -1}{q! (q+2)!} & = \Sigma_{q=0}^{\infty}\frac{ q^{2} +2q}{q! (q+2)!} - \Sigma_{q=0}^{\infty}\frac{1}{q! (q+2)!} \\
& =\Sigma_{q=1}^{\infty}\frac{ 1}{(q-1)! (q+1)!} - \Sigma_{q=0}^{\infty}\frac{1}{q! (q+2)!}\\
& =\Sigma_{q=0}^{\infty}\frac{1}{q! (q+2)!} - \Sigma_{q=0}^{\infty}\frac{1}{q! (q+2)!}\\
& =0. 
\end{split}
\end{equation*}
\end{proof}
 
 In addition to this proof, integrating the final formula of the curvature in the classical case confirms that the integral will be zero. This consists of the idea of our second proof that will be explained bellow.

The computations above from the relation (\ref{G-Btau}) to (\ref{finalseries}) showed that $\tau(R e^{f/2})$ is equal to the series in (\ref{finalseries}). So we should prove that this series vanish. If we put $h= 2(V + V^{-1}) = 2(e^{iy} +  e^{-iy})$, then we have the classical Gauss-Bonnet integral that is zero as the Euler characteristic for commutative torus is vanishes. If we show that this classical integral is exactly the same as series (\ref{finalseries}), we have shown the series (\ref{finalseries}) vanish. The Gauss-Bonnet theorem for the metric 
$$g=
\begin{pmatrix}
e^{h(y)} & 0 \\
0 & 1  \\
\end{pmatrix} 
$$
asserts that the integral

\begin{equation*}
\begin{split}
& \int_{\Pi^{2}} e^{\frac{h}{2}} \big \{  -\frac{1}{2} h''(y) - \frac{1}{4} (h'(y))^{2}   \big \}dxdy \\
= & 2 \pi \int_{\Pi} e^{\frac{h}{2}} \big \{  -\frac{1}{2} h''(y) - \frac{1}{4} (h'(y))^{2} \big \}dy
\end{split}
\end{equation*}
should be zero. We choose $h= 2(e^{iy} +  e^{-iy})$ and change the variable on the complex circle $e^{iy} \longrightarrow z$ to reach to
\\
\begin{equation*}
\begin{split}
& 2 \pi i \int_{|z|=1} e^{z+z^{-1}}  \big \{  -(z+z^{-1}) - (z-z^{-1})^{2} \big \} z^{-1}dz\\
 = & 2 \pi i \int_{|z|=1} e^{z+z^{-1}}  \big \{  -z-z^{-1}- z^{2}-z^{-2} + 2 \big \}z^{-1}dz.
\end{split}
\end{equation*}

The integrand has singular point only at $z=0$. Now due to the existence of coefficient $z^{-1}$, it will suffice to compute the constant term of the Laurent serries  of $e^{z+z^{-1}}$, $e^{z+z^{-1}}z$, $e^{z+z^{-1}}z^{-1}$,$e^{z+z^{-1}}z^{2}$ and $e^{z+z^{-1}}z^{-2}$.  We can then apply the residue formula $a_{-1} = \frac{1}{2 \pi i} \int f(z)dz$ to extract the integral value. A bit consideration reveals that these coefficients are obtained in the same lines of noncommutative computations in (\ref{G-Btau}) to  (\ref{finalseries}). So the resulting coefficients match with the series (\ref{finalseries}) which completes the proof.
\\
The main restriction in  Proposition \ref{G-Bsecondmetric} is that $f$ was chosen to be only a specific element $2(V+ V^{-1}) \in A_{\theta}^{\infty}$. However, we hope this point be useful when we relax the condition of commutativity of $f$ with its partial derivatives. In the following, we give a simple proof of the Gauss-Bonnet theorem of Proposition \ref{G-Bsecondmetric} for a self-adjoint element $f \in A_{\Theta}^{\infty}$ commuting with its partial derivatives.  
  
\begin{proposition} \label{G-Barbitraryf}
Given the metric
$\begin{pmatrix}
e^{f} & 0 \\
0 & 1 \\
\end{pmatrix}$
on the noncommutative two-torus $A_{\theta}$ where $f =f^{*}$ commutes with its partial derivatives, the Gauss-Bonnet theorem holds for the associated scalar curvature, $\tau(R \sqrt{detg}) = \tau(R e^{f/2})=0$
\end{proposition}

\begin{proof}
The tracial property of $\tau$ gives $\tau(e^{f/2} \partial^{2}f)= -\frac{1}{2} \tau(e^{f/2} (\partial_{2}f)^{2})$. So we will have

\begin{equation*}
\begin{split}
\tau (R \sqrt{detg}) & = \tau  \big(  e^{\frac{f}{2}}  \big\{ -  \partial_{2}^{2}f - \frac{1}{2} (\partial_{2}f)^{2} \big\}\big) \\
 & = \frac{1}{2} \tau(e^{f/2} (\partial_{2}f)^{2}) - \frac{1}{2} \tau(e^{f/2} (\partial_{2}f)^{2}) \\
 & =0.
\end{split}
\end{equation*}

\end{proof}

\begin{remark}
There is similar result in \cite{Pet1} on Proposition \ref{G-Barbitraryf} for a diagonal metric 
$\begin{pmatrix}
a & 0 \\
0 & b \\
\end{pmatrix}$
where the positive elements $a$ and $b$ commute with their partial derivatives, but not necessarily commute with each other. The other condition in that paper is that the Gauss-Bonnet theorem holds when $ab=1$. We consider a generalization of this metric in the next section.  
\end{remark}

In the following, we find all the self-adjoint elements in $A_{\Theta}^{\infty}$ which commute with their partial derivatives in addition to the central elements $Z(A_{\theta}^{\infty}) \simeq \mathbb{C}$.  

\begin{proposition}
The self-adjoint elements $\Sigma_{m \in \mathbb{Z}} (a_{mm}U^{m}V^{m} + \overline{a_{mm}} V^{-m}U^{-m})$, $ a_{mn}U^{m}V^{n} +\overline{a_{mn}}V^{-n}U^{-m}$ and the elements 
$$a_{mn}U^{m}V^{n} +\overline{a_{mn}}V^{-n}U^{-m} +  a_{kl}U^{k}V^{l} +\overline{a_{kl}}V^{-l}U^{-k}$$ 
in $A_{\theta}^{\infty}$ with the condition $kn=ml$ commute with their partial derivatives.
\end{proposition}
\begin{proof}
Each element $f=f^{*} \in A_{\theta}^{\infty}$ is of the form
 \begin{equation} \label{arbitraryselfadjf}
\Sigma_{(m,n) \in \mathbb{Z}^{2}} a_{mn}U^{m}V^{n} + \overline{a_{mn}} V^{-n}U^{-m}.
 \end{equation}
We first show that the general term of the series (\ref{arbitraryselfadjf}) commute with its partial derivatives. We take $f= a_{mn}U^{m}V^{n} +\overline{a_{mn}}V^{-n}U^{-m}$ and observe 

\begin{equation*}
\begin{split}
f \partial_{1}f  = & \big( a_{mn}U^{m}V^{n} +\overline{a_{mn}}V^{-n}U^{-m} \big) \big(m a_{mn}U^{m}V^{n} -m \overline{a_{mn}}V^{-n}U^{-m} \big) \\
= &   m a_{mn}^{2}U^{m}V^{n}U^{m}V^{n} - |a_{mn}|^{2}U^{m}V^{n}V^{-n}U^{-m}\\
& + m  |a_{mn}|^{2}V^{-n}U^{-m} U^{m}V^{n} - m \overline{a_{mn}}^{2} V^{-n}U^{-m}V^{-n}U^{-m}\\
= & m a_{mn}^{2}U^{m}V^{n}U^{m}V^{n} + m \overline{a_{mn}}^{2} V^{-n}U^{-m}V^{-n}U^{-m} = (\partial_{1}f) f ,
\end{split}
\end{equation*}
and in a similar way $f \partial_{2}f =(\partial_{2}f) f $. Now we consider the case when the product of arbitrary terms in the series(\ref{arbitraryselfadjf}) occurs. For 
$$f =  a_{mn}U^{m}V^{n} +\overline{a_{mn}}V^{-n}U^{-m} +  a_{kl}U^{k}V^{l} +\overline{a_{kl}}V^{-l}U^{-k},$$
we have
$$f \partial_{1}f =\big( a_{mn}U^{m}V^{n} +\overline{a_{mn}}V^{-n}U^{-m} +  a_{kl}U^{k}V^{l} +\overline{a_{kl}}V^{-l}U^{-k} \big) \times $$
$$\big( m a_{mn}U^{m}V^{n} -m \overline{a_{mn}}V^{-n}U^{-m} +  k a_{kl}U^{k}V^{l} -k \overline{a_{kl}}V^{-l}U^{-k} \big) $$
$$= k  a_{mn} a_{kl} U^{m}V^{n} U^{k}V^{l} -k a_{mn}\overline{a_{kl}} U^{m}V^{n}V^{-l}U^{-k} $$
$$+ k \overline{a_{mn}} a_{kl}V^{-n}U^{-m} U^{k}V^{l} - k \overline{a_{mn}} \overline{a_{kl}}V^{-n}U^{-m}  V^{-l}U^{-k}$$
$$+m  a_{mn} a_{kl} U^{k}V^{l}U^{m}V^{n} -m \overline{a_{mn}} a_{kl} U^{k}V^{l}V^{-n}U^{-m} $$ 
$$+ ma_{mn} \overline{a_{kl}} V^{-l}U^{-k}U^{m}V^{n} - m \overline{a_{mn}} \overline{a_{kl}}V^{-l}U^{-k} V^{-n}U^{-m}, $$

and
$$(\partial_{1}f) f = \big( a_{mn}U^{m}V^{n} +\overline{a_{mn}}V^{-n}U^{-m} +  a_{kl}U^{k}V^{l} +\overline{a_{kl}}V^{-l}U^{-k} \big) \times $$
$$\big( m a_{mn}U^{m}V^{n} -m \overline{a_{mn}}V^{-n}U^{-m} +  k a_{kl}U^{k}V^{l} -k \overline{a_{kl}}V^{-l}U^{-k} \big) $$

$$= m a_{mn} a_{kl} U^{m}V^{n} U^{k}V^{l} +m a_{mn}\overline{a_{kl}} U^{m}V^{n}V^{-l}U^{-k} $$
$$-m \overline{a_{mn}} a_{kl}V^{-n}U^{-m} U^{k}V^{l} -m  \overline{a_{mn}} \overline{a_{kl}}V^{-n}U^{-m}  V^{-l}U^{-k}$$
$$+ k  a_{mn} a_{kl} U^{k}V^{l}U^{m}V^{n} +k  \overline{a_{mn}} a_{kl} U^{k}V^{l}V^{-n}U^{-m}$$
$$-k a_{mn} \overline{a_{kl}} V^{-l}U^{-k}U^{m}V^{n} -k  \overline{a_{mn}} \overline{a_{kl}}V^{-l}U^{-k} V^{-n}U^{-m}. $$

The above two terms can be simplified by the commutation relation of the algebra to
$$f \partial_{1}f =\big\{ k  a_{mn} a_{kl} e^{-i\theta (kn) } +m a_{mn} a_{kl} e^{-i\theta (ml)} \big\} U^{m+k} V^{n+l}$$

\begin{equation} \label{leftfpartial}
 + \big\{ -k  a_{mn}\overline{a_{kl}} e^{-i\theta (-k)(n-l)} +m a_{mn}\overline{a_{kl}} e^{-i\theta (-l)(m-k)} \big\} U^{m-k} V^{n-l}  
\end{equation}
$$+ \big\{  k \overline{a_{mn}} a_{kl} e^{-i\theta (-n)(-m +k)} - m \overline{a_{mn}} a_{kl} e^{-i\theta (-m)(-n +l)} \big\} U^{-m+k} V^{-n+l}$$
$$+ \big\{  -k \overline{a_{mn}} \overline{a_{kl}} e^{-i \theta (kl + mn+nk)} -m \overline{a_{mn}}  \overline{a_{kl}} e^{-i \theta (kl + mn+ml)} \big\}U^{-m-k} V^{-n-l},$$

and to

$$(\partial_{1}f) f = \big\{ m  a_{mn} a_{kl} e^{-i\theta (kn) } +k a_{mn} a_{kl} e^{-i\theta (ml)} \big\} U^{m+k} V^{n+l}$$
\begin{equation} \label{rightfpartial}
 + \big\{ m  a_{mn}\overline{a_{kl}} e^{-i\theta (-k)(n-l)} -k a_{mn}\overline{a_{kl}} e^{-i\theta (-l)(m-k)} \big\} U^{m-k} V^{n-l}  
\end{equation}
$$+ \big\{  -m \overline{a_{mn}} a_{kl} e^{-i\theta (-n)(-m +k)} +k \overline{a_{mn}} a_{kl} e^{-i\theta (-m)(-n +l)} \big\} U^{-m+k} V^{-n+l}$$
$$+ \big\{  -m \overline{a_{mn}} \overline{a_{kl}} e^{-i \theta (kl + mn+nk)} -k \overline{a_{mn}} \overline{a_{kl}} e^{-i \theta (kl + mn+ml)} \big\}U^{-m-k} V^{-n-l}. $$

Now we check the equality of the relations (\ref{leftfpartial}) and (\ref{rightfpartial}). In order to have the equality of the first rows, we should have the equality of the coefficients $k a_{mn} a_{kl} e^{-i\theta (kn)} = k a_{mn} a_{kl} e^{-i\theta (ml)}$ which restrict us to the condition 

\begin{equation} \label{hpartialh=partialhh}
kn =ml.
\end{equation}
To have $f \partial_{1}f = (\partial_{1}f) f $, it suffices to change the role of $m$ with $n$ and $k$ with $l$ to reach to the same relation $lm= nk$. The condition $m=n$ and $k=l$ makes also the relation (\ref{hpartialh=partialhh}) hold. So the following $f$ commute with its canonical derivations

$$f= \Sigma_{m} a_{mm}U^{m}V^{m} + \overline{a_{mm}}V^{-m}U^{-m}.$$
For the condition $kn =ml$, we have the example
$$f =a_{23} U^{2} V^{3} +\overline{a_{23}} V^{-3} U^{-2} + a_{46} U^{4}V^{6} +\overline{a_{46}} V^{-6}U^{-4}.$$
\end{proof}
%...................................................................%
%...................................................................%
\subsection{A non-diagonal metric}
 Consider the metric
\begin{equation} \label{tmetric}
g=
\begin{pmatrix}
e^{f} & t \\
t & e^{-f} \\
\end{pmatrix},
\end{equation}
 
where $0 \leq t < 1$. This metric is Riemannian since we know every metric 
$\begin{pmatrix}
a & c \\
c^{*} & b \\ 
\end{pmatrix} $
is positive in the two dimensional matrix $C^{*}-$algebra on $A_{\Theta}^{\infty}$ if and only if we have $a \geq c^{*} b^{-1} c$ which is translated in this case to $e^{f} > t^{2} e^{f}$ or $(1-t^{2})e^{f} > 0$. The inverse of this metric is obtained to be
$g^{-1}= \frac{1}{1-t^{2}}
\begin{pmatrix}
e^{-f} & -t \\
-t & e^{f} \\
\end{pmatrix}
$. We note that this inverse metric is a two sided inverse. In spite of previous computations about the curvature, since the metric in this case is non-diagonal we should consider the formula involving all the Christoffel symbols. Besides, we know that due to the choose of Levi-Civita connection, these symbols are symmetric, $\Gamma_{ij}^{k} = \Gamma_{ji}^{k}$. 
\begin{proposition} \label{riemanniannondiagonal}
Given the Riemannian metric (\ref{tmetric}) on the vector field $X_{\theta}$ on $A_{\theta}^{\infty}$ and its associated Levi-Civita connection, the Gauss-Bonnet theorem holds for each $0 \leq t < 1$.
\end{proposition}

\begin{proof}
We have

\begin{equation} \label{connection}
\bigtriangledown_{\partial_{1}} \partial_{2} =\bigtriangledown_{\partial_{2}} \partial_{1} = \Gamma_{12}^{1} \partial_{1} + \Gamma_{12}^{2} \partial_{2},
\end{equation}

$$   \bigtriangledown_{\partial_{1}} \partial_{1} = \Gamma_{11}^{1} \partial_{1} + \Gamma_{11}^{2} \partial_{2}, \hspace{1 cm} \bigtriangledown_{\partial_{2}} \partial_{2} = \Gamma_{22}^{1} \partial_{1} + \Gamma_{22}^{2} \partial_{2}.$$
\\
We compute the Christoffel coefficients by means of the formula:
$$\Gamma_{ij}^{k} = \frac{1}{2} g^{kl} \big\{
\partial_{i}g_{jl} + \partial_{j}g_{il} - \partial_{k}g_{ij}
 \big\},$$
which gives the following relations:
$$\Gamma_{11}^{1} = \frac{1}{2} g^{11} \partial_{1}g_{11} - \frac{1}{2} g^{12} \partial_{2}g_{11}
= \frac{1}{2(1-t^{2})} e^{-f}\partial_{1}e^{f} + \frac{t}{2(1-t^{2})}\partial_{2}e^{f}, $$

$$\Gamma_{11}^{2} = \frac{1}{2} g^{21} \partial_{1}g_{11} - \frac{1}{2} g^{22} \partial_{2}g_{11}
=- \frac{t}{2(1-t^{2})}\partial_{1}e^{f} - \frac{1}{2(1-t^{2})} e^{f} \partial_{2}e^{f}, $$
and
$$\Gamma_{22}^{1} = -\frac{1}{2} g^{11} \partial_{1}g_{22} + \frac{1}{2} g^{12} \partial_{2}g_{22}
= -\frac{1}{2(1-t^{2})} e^{-f}\partial_{1}e^{-f} - \frac{t}{2(1-t^{2})}\partial_{2}e^{-f}, $$

\begin{equation} \label{christoffelsymbols}
\Gamma_{22}^{2} = -\frac{1}{2} g^{21} \partial_{1}g_{22} + \frac{1}{2} g^{22} \partial_{2}g_{22}
= \frac{t}{2(1-t^{2})} \partial_{1}e^{-f} + \frac{1}{2(1-t^{2})}e^{f}\partial_{2}e^{-f}, 
\end{equation}
and
$$\Gamma_{12}^{1} = \frac{1}{2} g^{11} \partial_{2}g_{11} + \frac{1}{2} g^{12} \partial_{1}g_{22}
= \frac{1}{2(1-t^{2})} e^{-f}\partial_{2}e^{f} - \frac{t}{2(1-t^{2})}\partial_{1}e^{-f}, $$

$$\Gamma_{12}^{2} = \frac{1}{2} g^{21} \partial_{2}g_{11} + \frac{1}{2} g^{22} \partial_{1}g_{22}
=- \frac{t}{2(1-t^{2})} \partial_{2}e^{f} + \frac{1}{2(1-t^{2})}e^{f} \partial_{1}e^{-f}.$$
\\
Now according to relations (\ref{connection}) and (\ref{christoffelsymbols}) we have
%..........................................................%
$$\bigtriangledown_{\partial_{1}} \partial_{2} = \Gamma_{12}^{1} \partial_{1} + \Gamma_{12}^{2} \partial_{2}$$
$$ = \big\{ \frac{1}{2(1-t^{2})} e^{-f}\partial_{2}e^{f} - \frac{t}{2(1-t^{2})}\partial_{1}e^{-f} \big\} \partial_{1} 
+\big\{- \frac{t}{2(1-t^{2})} \partial_{2}e^{f} + \frac{1}{2(1-t^{2})}e^{f} \partial_{1}e^{-f}  \big\}\partial_{2}$$
$$= \frac{1}{2(1-t^{2})} \big\{ e^{-f}\partial_{2}e^{f} -t \partial_{1}e^{-f} \big\} \partial_{1} 
+\big\{- t \partial_{2}e^{f} + e^{f} \partial_{1}e^{-f}  \big\}\partial_{2},$$

%..........................................................%

$$ \bigtriangledown_{\partial_{1}} \partial_{1} = \Gamma_{11}^{1} \partial_{1} + \Gamma_{11}^{2} \partial_{2}$$
$$=\big\{ \frac{1}{2(1-t^{2})} e^{-f}\partial_{1}e^{f} + \frac{t}{2(1-t^{2})}\partial_{2}e^{f} \big\} \partial_{1}
+ \big\{ - \frac{t}{2(1-t^{2})}\partial_{1}e^{f} - \frac{1}{2(1-t^{2})} e^{f} \partial_{2}e^{f} \big\} \partial_{2}
$$ 
$$= \frac{1}{2(1-t^{2})} \big\{ e^{-f}\partial_{1}e^{f} + t\partial_{2}e^{f} \big\} \partial_{1}
+ \big\{- t\partial_{1}e^{f} -e^{f} \partial_{2}e^{f} \big\} \partial_{2},$$

%..........................................................%

$$ \bigtriangledown_{\partial_{2}} \partial_{2} = \Gamma_{22}^{1} \partial_{1} + \Gamma_{22}^{2} \partial_{2}$$
$$ =  \big\{ -\frac{1}{2(1-t^{2})} e^{-f}\partial_{1}e^{-f} - \frac{t}{2(1-t^{2})}\partial_{2}e^{-f} \big\} \partial_{1}
+ \big\{  \frac{t}{2(1-t^{2})} \partial_{1}e^{-f} + \frac{1}{2(1-t^{2})}e^{f}\partial_{2}e^{-f} \big\} \partial_{2}
$$
$$= \frac{1}{2(1-t^{2})} \big\{ -e^{-f}\partial_{1}e^{-f} - t \partial_{2}e^{-f} \big\} \partial_{1} + \big\{ t \partial_{1}e^{-f} + e^{f}\partial_{2}e^{-f} \big\} \partial_{2}. $$

%..........................................................%

We use the following relations in computing the scalar curvature
$$R = g^{ij}R_{ij}, \hspace{1 cm} R_{ij}= R_{ikj}^{k}, $$ 
$$R_{11} = g^{12}R_{1211} + g^{22}R_{1212}, \hspace{1 cm} R_{12}= -g^{12}R_{2121} + g^{22}R_{1222},$$
$$R_{21} = g^{11}R_{2111} - g^{12}R_{1212}, \hspace{1 cm} R_{22} = g^{11}R_{2121} + g^{12}R_{2122}.$$
\\
So
$$ R = g^{11}g^{12}R_{1211} + g^{11}g^{22}R_{1212} - g^{12}g^{12}R_{2121}+ g^{12}g^{22}R_{1222}$$
$$ +g^{12}g^{11}R_{2111} - g^{12}g^{12}R_{1212} + g^{22}g^{11}R_{2121}+ g^{22}g^{12}R_{2122} $$

$$= \frac{1}{(1-t^{2})^{2}} \big\{ -te^{-f} R_{1211} + R_{1212} - t^{2} R_{2121} -te^{f}R_{1222}$$
$$-te^{-f}R_{2111}-t^{2}R_{1212} +R_{2121}- t e^{f} R_{2122} \big\},$$
which by symmetry $R_{ijkl} = -R_{jikl}$ becomes
\begin{equation} \label{curvaturesymbolic}
 R  = \frac{1}{(1-t^{2})^{2}} \big\{ (1-t^{2})(R_{1212} + R_{2121}) \big\}=\frac{1}{(1-t^{2})} \big\{R_{1212} + R_{2121} \big\} .
\end{equation}
Now we compute the curvature tensors in the relation (\ref{curvaturesymbolic})

%..........................................................%
$$R_{1212}= <(\bigtriangledown_{2}\bigtriangledown_{1} -\bigtriangledown_{1}\bigtriangledown_{2}) \partial_{1}, \partial_{2}> = A-B, $$
%..........................................................%

where
$$A= <\bigtriangledown_{2}\bigtriangledown_{1}\partial_{1}, \partial_{2}> =$$
$$=\frac{1}{2(1-t^{2})} <\bigtriangledown_{2} \big( \big\{ e^{-f}\partial_{1}e^{f} + t\partial_{2}e^{f} \big\} \partial_{1} +\big\{- t\partial_{1}e^{f} -e^{f} \partial_{2}e^{f} \big\} \partial_{2}\big), \partial_{2} >$$
$$= \frac{t}{2(1-t^{2})} \partial_{2} \big\{ e^{-f}\partial_{1}e^{f} + t\partial_{2}e^{f} \big\}+   \frac{1}{2(1-t^{2})} \partial_{2} \big\{- t\partial_{1}e^{f} -e^{f} \partial_{2}e^{f} \big\}e^{-f} $$
$$+ \frac{1}{2(1-t^{2})} \big\{ e^{-f}\partial_{1}e^{f} + t\partial_{2}e^{f} \big\} <\bigtriangledown_{2}\partial_{1}, \, \partial_{2}> $$
$$+\frac{1}{2(1-t^{2})} \big\{- t\partial_{1}e^{f} -e^{f} \partial_{2}e^{f} \big\} \big\} <\bigtriangledown_{2}\partial_{2}, \, \partial_{2}>, $$
\\
and
%..........................................................%
$$B = <\bigtriangledown_{1}\bigtriangledown_{2} \partial_{1}, \partial_{2}> $$
$$=  \frac{1}{2(1-t^{2})} <\bigtriangledown_{1} \big( \big\{ e^{-f}\partial_{2}e^{f} -t \partial_{1}e^{-f} \big\} \partial_{1} +\big\{- t \partial_{2}e^{f} + e^{f} \partial_{1}e^{-f} \big\}\partial_{2} \big),  \partial_{2}>$$

$$=  \frac{t}{2(1-t^{2})} \partial_{1} \big\{ e^{-f}\partial_{2}e^{f} -t \partial_{1}e^{-f} \big\} + \frac{1}{2(1-t^{2})} \partial_{1} \big\{- t \partial_{2}e^{f} + e^{f} \partial_{1}e^{-f} \big\}e^{-f} $$

$$+ \frac{1}{2(1-t^{2})} \big\{ e^{-f}\partial_{2}e^{f} -t \partial_{1}e^{-f} \big\} <\bigtriangledown_{1} \partial_{1}, \partial_{2}> 
+ \frac{1}{2(1-t^{2})}\big\{- t \partial_{2}e^{f} + e^{f} \partial_{1}e^{-f} \big\} <\bigtriangledown_{1} \partial_{2}, \partial_{2}> $$
\\
On the other hand, we have
%..........................................................%
$$R_{2121}= <(\bigtriangledown_{1}\bigtriangledown_{2} -\bigtriangledown_{2}\bigtriangledown_{1}) \partial_{2}, \partial_{1}> = A-B $$
%..........................................................%

$$A= <  \bigtriangledown_{1}\bigtriangledown_{2}\partial_{2}, \partial_{1}> $$
$$= \frac{1}{2(1-t^{2})} <\bigtriangledown_{1} \big( \big\{ -e^{-f}\partial_{1}e^{-f} - t \partial_{2}e^{-f} \big\} \partial_{1} 
+ \big\{ t \partial_{1}e^{-f} + e^{f}\partial_{2}e^{-f} \big\} \partial_{2} \big), \partial_{1}>$$

$$= \frac{1}{2(1-t^{2})} \partial_{1} \big\{ -e^{-f}\partial_{1}e^{-f} - t \partial_{2}e^{-f} \big\}e^{f} + \frac{t}{2(1-t^{2})} \partial_{1} \big\{ t \partial_{1}e^{-f} + e^{f}\partial_{2}e^{-f} \big\} $$
$$+  \frac{1}{2(1-t^{2})} \big\{ -e^{-f}\partial_{1}e^{-f} - t \partial_{2}e^{-f} \big\}  <\bigtriangledown_{1}\partial_{1}, \partial_{1}>$$
$$+ \frac{1}{2(1-t^{2})} \big\{ t \partial_{1}e^{-f} + e^{f}\partial_{2}e^{-f} \big\} <\bigtriangledown_{1}\partial_{2}, \partial_{1}>,$$
%..........................................................%
\\
$$B= <\bigtriangledown_{2}\bigtriangledown_{1}\partial_{2}, \partial_{1}> $$
$$= \frac{1}{2(1-t^{2})} <\bigtriangledown_{2} \big( \big\{ e^{-f}\partial_{2}e^{f} -t \partial_{1}e^{-f} \big\} \partial_{1} 
+\big\{- t \partial_{2}e^{f} + e^{f} \partial_{1}e^{-f}  \big\}\partial_{2} \big),\partial_{1}>$$

$$= \frac{1}{2(1-t^{2})} \partial_{2} \big\{ e^{-f}\partial_{2}e^{f} -t \partial_{1}e^{-f} \big\}e^{f}
+\frac{t}{2(1-t^{2})} \partial_{2} \big\{- t \partial_{2}e^{f} + e^{f} \partial_{1}e^{-f}  \big\}$$

$$+\frac{1}{2(1-t^{2})}\big\{ e^{-f}\partial_{2}e^{f} -t \partial_{1}e^{-f} \big\} <\bigtriangledown_{2}\partial_{1}, \partial_{1}>$$
$$+ \frac{1}{2(1-t^{2})} \big\{- t \partial_{2}e^{f} + e^{f} \partial_{1}e^{-f}  \big\}<\bigtriangledown_{2}\partial_{2}, \partial_{1}>.$$
%..........................................................%

Since $det g=1-t^{2}$, the relation (\ref{curvaturesymbolic}) gives 
$$\tau(R \sqrt{detg}) = (1-t^{2})^{\frac{1}{2}} \tau(R)= (1-t^{2})^{\frac{-1}{2}} \tau(R_{1212}+ R_{2121}).$$
According to the above relations for the curvature tensors $R_{1212}$ and $R_{2121}$,
$$2(1-t^{2})\tau(R_{1212})=  \tau \big( \partial_{2} \big\{- t\partial_{1}e^{f} -e^{f} \partial_{2}e^{f} \big\}e^{-f} + \frac{1}{2} \big\{ e^{-f}\partial_{1}e^{f} + t\partial_{2}e^{f} \big\} \partial_{1}e^{-f} $$
$$+\frac{1}{2} \big\{- t\partial_{1}e^{f} -e^{f} \partial_{2}e^{f} \big\}\partial_{2}e^{-f} -\partial_{1} \big\{- t \partial_{2}e^{f} + e^{f} \partial_{1}e^{-f} \big\}e^{-f}$$

$$+\frac{1}{2} \big\{ e^{-f}\partial_{2}e^{f} -t \partial_{1}e^{-f} \big\}\partial_{2}e^{f} 
- \frac{1}{2}\big\{- t \partial_{2}e^{f} + e^{f} \partial_{1}e^{-f} \big\} \partial_{1}e^{-f}\big),$$
%..........................................................%
and
$$2(1-t^{2}) \tau(R_{2121}) = \tau \big( 
\partial_{1} \big\{ -e^{-f}\partial_{1}e^{-f} - t \partial_{2}e^{-f} \big\}e^{f}
+ \frac{1}{2} \big\{ -e^{-f}\partial_{1}e^{-f} - t \partial_{2}e^{-f} \big\} \partial_{1}e^{f}$$
$$ + \frac{1}{2} \big\{ t \partial_{1}e^{-f} + e^{f}\partial_{2}e^{-f} \big\} \partial_{2}e^{f}
-\partial_{2} \big\{e^{-f}\partial_{2}e^{f} -t \partial_{1}e^{-f} \big\}e^{f}$$
$$-\frac{1}{2}\big\{ e^{-f}\partial_{2}e^{f} -t \partial_{1}e^{-f} \big\}\partial_{2}e^{f}
+\frac{1}{2} \big\{- t \partial_{2}e^{f} + e^{f} \partial_{1}e^{-f}  \big\}\partial_{1}e^{-f}\big).$$
%..........................................................%
%..........................................................%
Since the derivations vanish under the trace, it is simplified to
$$2(1-t^{2})\tau(R_{1212})=  \tau \big(  t(\partial_{1}e^{f})(\partial_{2}e^{-f}) + e^{f} (\partial_{2}e^{f})(\partial_{2}e^{-f}) + \frac{1}{2} e^{-f}(\partial_{1}e^{f})(\partial_{1}e^{-f})$$
$$ +\frac{1}{2} t(\partial_{2}e^{f})(\partial_{1}e^{-f}) -\frac{1}{2}t(\partial_{1}e^{f})(\partial_{2}e^{-f}) - \frac{1}{2}e^{f} (\partial_{2}e^{f})(\partial_{2}e^{-f})$$
$$ - t (\partial_{2}e^{f})(\partial_{1}e^{-f}) + e^{f} (\partial_{1}e^{-f})(\partial_{1}e^{-f})
+ \frac{1}{2} e^{-f}(\partial_{2}e^{f})(\partial_{2}e^{f})$$
$$ -\frac{1}{2}t (\partial_{1}e^{-f})(\partial_{2}e^{f}) 
+\frac{1}{2} t (\partial_{2}e^{f})(\partial_{1}e^{-f}) -\frac{1}{2} e^{f} (\partial_{1}e^{-f}) (\partial_{1}e^{-f}) \big),$$
%..........................................................%
and
$$2(1-t^{2}) \tau(R_{2121}) = \tau \big( 
e^{-f}(\partial_{1}e^{-f})(\partial_{1}e^{f}) + t (\partial_{2}e^{-f})(\partial_{1}e^{f})
- \frac{1}{2}e^{-f}(\partial_{1}e^{-f})(\partial_{1}e^{f})$$
$$ - \frac{1}{2}t (\partial_{2}e^{-f})(\partial_{1}e^{f}) + \frac{1}{2} t (\partial_{1}e^{-f})(\partial_{2}e^{f}) +\frac{1}{2} e^{f}(\partial_{2}e^{-f})(\partial_{2}e^{f})$$
$$ +e^{-f}(\partial_{2}e^{f})(\partial_{2}e^{f}) -t (\partial_{1}e^{-f})(\partial_{2}e^{f}) -\frac{1}{2}e^{-f}(\partial_{2}e^{f})(\partial_{2}e^{f}) $$
$$+\frac{1}{2}t (\partial_{1}e^{-f})(\partial_{2}e^{f})
-\frac{1}{2}t (\partial_{2}e^{f})(\partial_{1}e^{-f}) + \frac{1}{2}e^{f} (\partial_{1}e^{-f})(\partial_{1}e^{-f})\big).$$
%..........................................................%
\\
By separating the coefficients of similar terms we obtain

\begin{equation*}
\begin{split}
= & \tau \big( \{ 1-\frac{1}{2} + 1 - \frac{1}{2}\} t(\partial_{1}e^{f})(\partial_{2}e^{-f}) \\
& + \{\frac{1}{2} -1 -\frac{1}{2}+\frac{1}{2}+\frac{1}{2}-1+\frac{1}{2}-\frac{1}{2} \}t (\partial_{1}e^{-f})(\partial_{2}e^{f})\\
& +\{ 1-\frac{1}{2} -\frac{1}{2}+\frac{1}{2} -1 +\frac{1}{2} \} e^{f}(\partial_{2}e^{f})(\partial_{2}e^{-f})\\
& + \{\frac{1}{2} -1 +\frac{1}{2} +1 -\frac{1}{2}  -\frac{1}{2} \} e^{-f}(\partial_{1}e^{f})(\partial_{1}e^{-f}) \big),
\end{split}
\end{equation*}

that will be zero since $\tau(t(\partial_{1}e^{f})(\partial_{2}e^{-f})) = \tau(t (\partial_{1}e^{-f})(\partial_{2}e^{f}))$. This proves the proposition.
\end{proof}
\begin{remark}
Although we have considered the sum of $R_{1212}$ and $R_{2121}$ under the trace, but in fact we have both $\tau(R_{1212})=0$ and $\tau(R_{2121})=0$. However this does not mean that $R_{1212}= R_{2121}$ even when $t=0$ which was considered in \cite{Pet1}.
\end{remark}
If we substitute in Proposition \ref{Levi-Civita} the Riemannian metric by a hermitian metric, we can still apply the connection in this proposition, but the connection is not the  unique associated connection, called Chern connection. However we still have the same result of Proposition \ref{riemanniannondiagonal}. We note that by using a hermitian metric, we have already omitted the condition $(4)$ of the Riemannian metric in Definition \ref{Riemannianmetric}.
\begin{proposition} \label{hermitiannondiagonal}
Given the hermitian metric
$\begin{pmatrix}
e^{f} & \alpha \\
\overline{\alpha} & e^{-f} \\
\end{pmatrix} $
for $\alpha \in \mathbb{C}$ and $0 \leq |\alpha| < 1$ and also the Riemannian connection given in \ref{Levi-Civita}, the Gauss-Bonnet theorem holds.
\end{proposition}

\begin{proof}
The inverse metric is given by 
$
\frac{1}{1- |\alpha|^{2}}
\begin{pmatrix}
e^{-f} & \alpha \\
\overline{\alpha} & e^{f} \\
\end{pmatrix}
$.  The proof is in the same line of Proposition \ref{riemanniannondiagonal}, unless the entry $g_{21}=t$ should be replaced by $\overline{\alpha}$. 
\end{proof}

\section{Obstructions to the Gauss-Bonnet theorem} \label{SecGBFails}
In this section we discuss about the cases when the Gauss-Bonnet theorem does not hold. The methods we use are higher order perturbations as well as using projections. 
\subsection{Failing the G-B by means of perturbations}
By relation (\ref{GBnondiagonal}), the Gauss-Bonnet trace for the metric
$g=
\begin{pmatrix}
e^{f} & 0 \\
0 & 1 \\
\end{pmatrix}
$
is given by
\begin{equation} \label{GBtau2}
\tau(R\sqrt{g})=\frac{1}{2} \tau\big( (\partial_{2}e^{-\frac{f}{2}}) (\partial_{2}e^{f}) + (\partial_{2}e^{\frac{f}{2}})(\partial_{2}e^{f}) e^{-f} \big).
\end{equation}
Now we perturb the metric component $g_{11}$ by the time parameter to $e^{tf}$ to reach to 
\begin{equation} 
\Omega_{f}(t) :=\frac{1}{2} \tau\big( (\partial_{2}e^{-\frac{tf}{2}}) (\partial_{2}e^{tf}) + (\partial_{2}e^{\frac{tf}{2}})(\partial_{2}e^{tf}) e^{-tf} \big).
\end{equation}
where we see $\Omega_{f}(1) = \tau(R \sqrt{g})$. The trace $\Omega_{f}(t): \mathbb{R} \longrightarrow \mathbb{C}$ has formal power series 
\begin{equation} \label{powerseries}
\Omega_{f}(t) = \Omega_{f}(0) + \Omega_{f}^{'}(0)t+ \Omega_{f}^{''}(0)t^{2}+\Omega_{f}^{'''}(0)t^{3} + \Omega_{f}^{(4)}(0)t^{4} + ... \hspace{1 cm}.
\end{equation}
While this series is convergent at $t =0$, it gives non-trivial cases in the conformal flat metrics (\ref{conformametrictwodim}) for which the G-B    term (\ref{GBtau2}) does not vanish.
\begin{proposition} \label{twosidedcorespondence}
The Gauss-Bonnet theorem holds for any self-adjoint element $f \in A_{\Theta}^{\infty}$, i.e. $\Omega_{f}(1) =0$ if and only if for any $t \in \mathbb{R}$, $\Omega_{f}(t) =0$. 
\end{proposition}
\begin{proof}
Since the space of all elements $tf$, $f$ self-adjoint and $t \neq 0$ contains all self-adjoint elements in the noncommutative torus.
\end{proof}
Now we state the main result of this section.
\begin{proposition} \label{FailGBbyPerturbation}
Given any self-adjoint element $f \in A_{\Theta}^{\infty}$ with the condition 
\begin{equation} \label{order4tracecondition}
\tau \big(f^{2}(\partial_{2}f)(\partial_{2}f)-f(\partial_{2}f)f(\partial_{2}f)  \big) \neq 0,
\end{equation}
there is an interval $|t -\epsilon| < \delta$ for any $\epsilon \in \mathbb{R}$, such that for any $tf$, $t \in B_{\delta}(\epsilon)$, the Gauss-Bonnet theorem violates. I.e. $\Omega_{f}(t) \neq 0$. 
\end{proposition}
\begin{proof}
Out of the domain of the convergence of the series $\Omega_{f}(t)$ it does not vanish, so the proposition holds. In the domain of convergence, by considering the power series (\ref{powerseries})
$$ \tau \big\{ \big( -t \frac{\partial_{2}f}{2}+ t^{2} \frac{\partial_{2}f^{2}}{4 \times 2!} - t^{3} \frac{\partial_{2}f^{3}}{8 \times 3!}+ ... \big) \big(t \partial_{2} f + t^{2} \frac{\partial_{2}f^{2}}{ 2!} + t^{3} \frac{\partial_{2}f^{3}}{ 3!}+ ... \big)$$

$$+ \big( t \frac{\partial_{2}f}{2}+ t^{2} \frac{\partial_{2}f^{2}}{4 \times 2!}+ t^{3} \frac{\partial_{2}f^{3}}{8 \times 3!}+ ... \big)  \big(t \partial f + t^{2} \frac{\partial_{2}f^{2}}{ 2!} + t^{3} \frac{\partial_{2}f^{3}}{ 3!}+ ...\big) \big(1 - tf + t^{2}\frac{f^{2}}{2!} - t^{3}\frac{f^{3}}{3!} +...  \big) \big\}$$
 that $\Omega_{f}(0)$ $= \Omega_{f}^{'}(0)$ $=\Omega_{f}^{''}(0)$ $ = \Omega_{f}^{'''}(0)= 0$. However about $\Omega_{f}^{(4)}(0)$ we have

\begin{equation} \label{order4trace}
\begin{split}
\Omega_{f}^{(4)}(0) = &  t^{4} \tau \big\{
- \frac{(\partial_{2}f)(\partial_{2}f^{3})}{2 \times 3!}+  \frac{(\partial_{2}f^{2})(\partial_{2}f^{2})}{4 \times 2! \times 2!} -  \frac{(\partial_{2}f^{3})(\partial_{2}f)}{8 \times 3!}
\big\} \\
& +t^{4} \tau \big\{
\frac{(\partial_{2}f)(\partial_{2}f^{3})}{2 \times 3!} +  \frac{(\partial_{2}f^{2})(\partial_{2}f^{2})}{4 \times 2! \times 2!} + \frac{(\partial_{2}f^{3})(\partial_{2}f)}{8 \times 3!}\\
&- \frac{(\partial_{2}f)(\partial_{2}f^{2})f}{2 \times 2!}- \frac{(\partial_{2}f^{2})(\partial_{2}f)f}{4 \times 2!} + \frac{(\partial_{2}f)(\partial_{2}f)f^{2}}{2 \times 2!}
\big\}\\
= & t^{4} \tau \big\{
- \frac{(\partial_{2}f)(\partial_{2}f^{3})}{2 \times 3!}+  \frac{(\partial_{2}f^{2})(\partial_{2}f^{2})}{4 \times 2! \times 2!} -  \frac{(\partial_{2}f^{3})(\partial_{2}f)}{8 \times 3!}
\big\} \\
& +t^{4} \tau \big\{ 
\frac{(\partial_{2}f^{2})(\partial_{2}f^{2})}{2 \times 2! \times 2!}  - \frac{(\partial_{2}f)(\partial_{2}f^{2})f}{2 \times 2!}- \frac{(\partial_{2}f^{2})(\partial_{2}f)f}{4 \times 2!} + \frac{(\partial_{2}f)(\partial_{2}f)f^{2}}{2 \times 2!}
\big\}\\
= & 3 \tau \big( f^{2}(\partial_{2}f)(\partial_{2}f)-f(\partial_{2}f)f(\partial_{2}f)  \big).
\end{split}
\end{equation}

So there is at least an $\epsilon$ in the domain of convergence that $\Omega_{f}(\epsilon) \neq 0$. And since any function is analytic, we induce that there is an interval $|t -\epsilon| < \delta$ in the domain of convergence such that for any $t$ in this interval $\Omega_{f}(t) \neq 0$. This completes the proof.
\end{proof}
The trace in relation (\ref{order4trace}) does not necessarily vanish for arbitrary self-adjoint $f \in A_{\theta}^{\infty}$. For example, the trace of $o(t^{4})$ in \ref{order4trace} for $f= U+ U^{-1} + V + V^{-1}$ is equal to
$$\frac{-6}{8} + \frac{-2e^{i \theta} -2e^{-i \theta}-2}{8}.$$

\subsection{Failing the G-B for projections}
We have shown in Proposition \ref{FailGBbyPerturbation} that for any $f = f^{*} \in A_{\Theta}^{\infty}$ the Gauss-Bonnet theorem does not hold for $tf$, $t$
in some interval provided the condition
\begin{equation} \label{order4tracecondition}
\tau \big(f^{2}(\partial_{2}f)(\partial_{2}f)-f(\partial_{2}f)f(\partial_{2}f)  \big) \neq 0.
\end{equation}
In this part, we give a class of self-adjoint elements that satisfiy the condition \ref{order4tracecondition} and so the Gauss-Bonnet theorem does not hold for them. We consider the projective elements $p= p^{*} = p^{2}$ and show that they satisfy the condition (\ref{order4tracecondition}). For such a $p$ we have 

\begin{equation*}
\begin{split}
\tau \big(f^{2}(\partial_{2}f)(\partial_{2}f)\big) & = \tau \big(f^{2}(\partial_{2}f^{2})(\partial_{2}f)\big) \\
& = \tau \big(f^{2}(\partial_{2}f)(\partial_{2}f)\big) + \tau \big(f(\partial_{2}f)f(\partial_{2}f)\big),
\end{split}
\end{equation*}
so
\begin{equation*}
\tau \big(f(\partial_{2}f)f(\partial_{2}f)\big) =0,
\end{equation*}
which gives

\begin{equation*} 
\begin{split}
\tau \big(f^{2}(\partial_{2}f)(\partial_{2}f)\big) - \tau\big( f(\partial_{2}f)f(\partial_{2}f)\big) & = \tau \big(f^{2}(\partial_{2}f)(\partial_{2}f)\big) \\
 & = \tau \big((\partial_{2}f) f^{2}(\partial_{2}f)\big) \\
&  = \tau \big( \big((\partial_{2}f) f\big) \big((\partial_{2}f)f\big)^{*}\big) \geq 0.
\end{split}
\end{equation*}

Sine the trace $\tau$ is faithful, the last term does not vanish unless we have $(\partial_{2}f)f = 0$. To have $(\partial_{2}f)f =- (\partial_{2}f)^{*}f =0$, it requires its constant term
$$\tau \big( \Sigma_{mn} n \overline{a_{mn}}V^{-n}U^{-m} \Sigma_{mn} a_{mn} U^{m} V^{n} \big) = \Sigma_{mn} n |a_{mn}|^{2}  $$ 
should be zero. This does not hold unless we have $p = \Sigma_{m} a_{m} U^{m}$ for which we know the G-B theorem holds since it commutes with its partial derivative. We summarize the above result as follows

\begin{proposition} \label{GBFailure}
For any projection $p \in A_{\Theta}^{\infty}$ which does not commute with its partial derivatives, the Gauss-Bonnet theorem does not hold for $tp$, for $t$ in some interval as given in Proposition \ref{FailGBbyPerturbation}. 
\end{proposition}
\begin{proof}
Given such a projection, the above computations shows that it satisfies the condition \ref{order4tracecondition} which by relation $(\ref{GBtau2})$ means that the Gauss-Bonnet term is non-zero.
\end{proof}
\begin{example}{Power-Riffel Projections}
\\
Consider the noncommutative torus as a cross-product algebra of $C(S^{1})$ by irrational rotations. Consider the noncommutative torus  $A_{\Theta}$ generated by $U$ and $V$ where $U, V: L^{2}(S^{1}) \longrightarrow L^{2}(S^{1})$ are unitary operators given by 
$$Uf(x) = e^{2 \pi i x}f(x), \hspace{1 cm} Vf(x) = f(x+ \theta).$$ 
A Powers-Rieffel projection is of the form 
$$ e = f_{-1}(U) V^{-1} + f_{0} + f_{1}(U) V.$$
where coefficient of $V$ for these projections is arbitrary. So according to the statement above, they satisfy in the condition of Proposition \ref{GBFailure} and the G-B fails about them for any $te$, $t$ in an interval.  
\end{example}

\end{document}